\documentclass[11pt]{article}

\title{Elementary spectral invariants and three-dimensional Reeb dynamics}
\author{Michael Hutchings\footnote{Partially supported by NSF grant DMS-2404790.}}
\date{}

\usepackage{amssymb}
\usepackage{latexsym}
\usepackage{amsmath}
\usepackage{amsthm}
\usepackage{amscd}
\usepackage{color}

\numberwithin{equation}{section}

\newtheorem{theorem}{Theorem}[section]

\newtheorem{proposition}[theorem]{Proposition}
\newtheorem{corollary}[theorem]{Corollary}
\newtheorem{lemma}[theorem]{Lemma}
\newtheorem{lemma-definition}[theorem]{Lemma-Definition}

\theoremstyle{definition}
\newtheorem{definition}[theorem]{Definition}

\newtheorem{remark}[theorem]{Remark}

\newtheorem{example}[theorem]{Example}

\newcommand{\floor}[1]{\left\lfloor #1 \right\rfloor}
\newcommand{\ceil}[1]{\left\lceil #1 \right\rceil}

\newcommand{\eqdef}{\;{:=}\;}

\newcommand{\C}{{\mathbb C}}

\newcommand{\R}{{\mathbb R}}

\newcommand{\Z}{{\mathbb Z}}

\newcommand{\op}{\operatorname}

\newcommand{\bpm}{\begin{pmatrix}}
\newcommand{\epm}{\end{pmatrix}}

\renewcommand{\epsilon}{\varepsilon}

\begin{document}

\maketitle

\begin{abstract}
We survey various recent results on the existence and properties of periodic orbits of Reeb vector fields in three dimensions. We give an introduction to the ``elementary spectral invariants'' of contact three-manifolds, and we explain how they can be used to prove some of these results. (The remaining results can be proved using spectral invariants from embedded contact homology, of which the elementary spectral invariants are a simplification.) We then review the ``alternative ECH capacities'' of symplectic four-manifolds, and explain how these can be modified to define the elementary spectral invariants. This document is a set of lecture notes for a minicourse taught by the author at the CIME school on Symplectic Dynamics and Topology in Cetraro in June 2025.
\end{abstract}

\setcounter{tocdepth}{2}

\tableofcontents

\section{Reeb dynamics in three dimensions}
\label{sec:intro}

We begin by recalling some recent results in three-dimensional Reeb dynamics that have been proved using ECH spectral invariants and/or the elementary spectral invariants of contact three-manifolds.

\subsection{Reeb vector fields}
\label{sec:Reebvf}

Let $Y$ be a $(2n-1)$-dimensional smooth manifold. A {\bf contact form\/} on $Y$ is a $1$-form $\lambda$ on $Y$ such that $\lambda\wedge (d\lambda)^{n-1}\neq 0$ everywhere. If $\lambda$ is a contact form, then there is a unique vector field $R$ on $Y$, called the {\bf Reeb vector field\/}, such that $d\lambda(R,\cdot)=0$ and $\lambda(R)=1$. A contact form $\lambda$ also determines a rank $2n-2$ subbundle $\xi=\op{Ker}(\lambda)$ of $TY$, which is a totally nonintegrable distribution. The restriction of $d\lambda$ to each fiber of $\xi$ is a linear symplectic form. A rank $2n-2$ subbundle of $TY$ that can be obtained as the kernel of a contact form is called\footnote{Note that a contact form gives an orientation of $Y$ by the $2n-1$ form $\lambda\wedge(d\lambda)^{n-1}$. A contact structure given by the kernel of a globally defined contact form is sometimes called a {\bf co-oriented\/} contact structure. Sometimes a ``contact structure'' is defined to be a rank $2n-2$ subbundle which is only {\em locally\/} the kernel of a contact form.} a {\bf contact structure\/} on $Y$.

There are two basic examples of contact forms. To introduce the first example, consider $\R^{2n}$ with coordinates $x_1,\ldots,x_n,y_1,\ldots,y_n$. The standard Liouville form on $\R^{2n}$ is a $1$-form defined by
\begin{equation}
\label{eqn:standardLiouvilleform}
\lambda_0=\frac{1}{2}\sum_{i=1}^n\left(x_i\,dy_i-y_i\,dx_i\right).
\end{equation}
Note that $d\lambda_0=\omega_0$, where $\omega_0$ is the standard symplectic form
\[
\omega_0=\sum_{i=1}^ndx_i\wedge dy_i.
\]
Now let $Y$ be a hypersurface in $\R^{2n}$. Suppose that $Y$ is transverse to the radial vector field
\[
\rho = \sum_{i=1}^n\left(x_i\frac{\partial}{\partial x_i} + y_i\frac{\partial}{\partial y_i}\right).
\]
(If $Y$ is also compact and connected, then by a slight abuse of terminology\footnote{If $Y$ is a star-shaped hypersurface in this sense, then the $2n$-dimensional compact region $X\subset\R^{2n}$ bounded by $Y$ is star-shaped in the classical sense, namely the intersection of $X$ with any line through the origin is an interval. However the converse is not true: If $X$ is a $2n$-dimensional region which is star-shaped in the classical sense and has smooth boundary, the boundary might not be transverse to the radial vector field.}, symplectic geometers often refer to $Y$ as a {\bf star-shaped hypersurface\/}.)
Then the restriction $\lambda=\lambda_0\big|_Y$ is a contact form on $Y$. Moreover, the associated Reeb vector field $R$ is related to Hamiltonian dynamics as follows. Let $H:\R^{2n}\to \R$ be a smooth function for which $Y$ is a regular level set, and let $X_H$ denote the associated Hamiltonian vector field on $\R^{2n}$ defined by
\[
\omega_0(X_H,\cdot) = dH.
\]
Then $X_H$ is tangent to $Y$, and there is a nonvanishing smooth function $f:Y\to\R$ such that on $Y$ we have
\begin{equation}
\label{eqn:ReebHamiltonian}
R=fX_H.
\end{equation}
Later, whenever we refer to a star-shaped hypersurface $Y$, it is always understood that we equip it with the contact form given by the restriction of $\lambda_0$.

For the second example, let $N$ be an $n$-dimensional smooth manifold with a Riemannian\footnote{This example also generalizes to Finsler metrics.} metric $g$. Let $Y$ be the unit cotangent bundle of $N$, namely
\[
Y =\left\{(q,p)\;\big|\; q\in N,\; p\in T_q^*N,\; \|p\|_g=1\right\} \subset T^*N.
\]
There is a canonical $1$-form $\lambda$ on $T^*N$: If $q\in N$, $p\in T_q^*N$, and $V\in T_{(q,p)}T^*N$, then $\lambda(V)=p(\pi_*V)$, where $\pi:T^*N\to N$ is the projection. If $q_1,\ldots,q_n$ are local coordinates in an open set $U\subset N$, and $q_1,\ldots,q_n,p_1,\ldots,p_n$ are the resulting local coordinates on $\pi^{-1}(U)\subset T^*N$, then in these coordinates we have
\[
\lambda=\sum_{i=1}^np_i\,dq_i.
\]
Now the restriction $\lambda|_Y$ is a contact form on $Y$. (The same is true more generally if $Y$ is transverse to each fiber of $T^*N\to N$ and the intersection of $Y$ with each fiber is a star-shaped hypersurface.) Moreover, under the diffeomorphism $T^*N\simeq TN$ induced by the metric $g$, the associated Reeb vector field $R$ corresponds to the geodesic flow on the unit tangent bundle of $N$. Thus trajectories of the Reeb vector field are equivalent to unit speed geodesics on $N$.

The above examples provide some motivation for studying the dynamics of Reeb vector fields, but there are many more examples. A theorem of Thurston-Winkelnkemper \cite{thurwink} asserts that every closed orientable three-manifold admits a contact structure, and today there are many more ways to construct contact structures in three dimensions. A more recent result of Borman-Eliashberg-Murphy \cite{bem} implies that any odd-dimensional manifold admits a contact structure, as long as it admits an ``almost contact structure'', which is determined by algebraic topology.

It is also worth noting that if $\lambda$ is a contact form on a $(2n-1)$-dimensional manifold $Y$, and if $f:Y\to\R$ is a smooth function, then $e^f\lambda$ is also a contact form on $Y$ with the same associated contact structure. If $f$ is a constant function, then the Reeb vector field associated to $e^f\lambda$ is a scaling of the Reeb vector field associated to $\lambda$. If $f$ is not constant, then the Reeb vector fields of $\lambda$ and $e^f\lambda$ are not parallel and can have very different dynamics.

\subsection{The Weinstein conjecture in three dimensions}

Let $\lambda$ be a contact form on $Y$. A basic aspect of the dynamics of the associated Reeb vector field $R$ is given by its periodic orbits, which we call ``Reeb orbits'' for short. To be precise, a {\bf Reeb orbit\/} is a map $\gamma:\R/T\Z\to Y$, for some $T>0$, such that $\gamma'(t)=R(\gamma(t))$ for all $t$. We declare two such maps to be equivalent if they differ by precomposition with a translation of the domain $\R/T\Z$. The Reeb orbit $\gamma$ is {\bf simple\/} if the map $\gamma$ is injective. A simple Reeb orbit is equivalent to a compact connected one-dimensional submanifold of $Y$ which is everywhere tangent to the Reeb vector field\footnote{In some literature, all Reeb orbits are understood to be simple. We need to be pedantic and consider non-simple Reeb orbits because of the way that Reeb orbits arise in the spectral invariants that we will introduce later.}. For example, when $Y$ is the unit cotangent bundle of a smooth manifold $N$ with a Riemannian metric, (simple) Reeb orbits correspond to (prime) oriented closed geodesics\footnote{In the context of closed geodesics, the adjective ``simple'' generally refers to geodesics that do not have self-intersections in $N$.} in $N$.

For any Reeb orbit $\gamma:\R/T\Z\to Y$, there is a unique positive integer $d$ such that the map $\gamma$ factors as
\[
\R/T\Z \longrightarrow \R/(T/d)\Z \stackrel{\overline{\gamma}}{\longrightarrow} Y
\]
where $\overline{\gamma}$ is a simple Reeb orbit. We refer to the positive integer $d$ as the {\bf covering multiplicity\/} of $\gamma$; the Reeb orbit $\gamma$ is simple if and only if $d=1$. We define the {\bf symplectic action\/} of $\gamma$, denoted by $\mathcal{A}(\gamma)$, to be the period
\[
\mathcal{A}(\gamma) = T=\int_{\R/T\Z}\gamma^*\lambda.
\]

With these definitions out of the way, the first question one might ask is: Do Reeb orbits exist? The {\bf Weinstein conjecture\/}, formulated by Weinstein \cite{weinsteinconjecture} in the 1970s, asserts that for every contact form on a closed manifold, there exists at least one Reeb orbit. 

The Weinstein conjecture is known for the two basic examples in \S\ref{sec:Reebvf}, and these were motivations for the conjecture. For the first example, a theorem of Rabinowitz \cite{rabinowitz} asserts that every star-shaped hypersurface in $\R^{2n}$ has at least one Reeb orbit. We remark that a folk conjecture asserts that every star-shaped hypersurface in $\R^{2n}$ has at least $n$ distinct simple Reeb orbits. This was recently proved by Cineli-Ginzburg-Gurel \cite{cgg} under the assumption that the star-shaped hypersurface $Y$ is dynamically convex\footnote{A star-shaped hypersurface $Y\subset \R^{2n}$ is {\bf dynamically convex\/} if every Reeb orbit has Conley-Zehnder index at least $n+1$ with respect to a global trivialization of the contact structure.}, which as in \cite{hwz98} includes the case when $Y$ is strictly convex.

For the second example, the classical Lyusternik-Fet theorem asserts that every closed Riemannian manifold has at least one closed geodesic. And there are many results asserting the existence of more than one prime closed geodesic. For example, any Riemannian metric on a closed surface has infinitely many prime closed geodesics; in positive genus one can prove this by minimizing length in infinitely many prime free homotopy classes, and Bangert \cite{bangert} and Franks \cite{franks} proved this in genus zero.

The three-dimensional case of the Weinstein conjecture was proved by Taubes \cite{taubesweinstein} in 2006. In fact, Taubes proved the following stronger statement:

\begin{theorem}[Taubes \cite{taubesweinstein}]
\label{thm:taubesweinstein}
Let $Y$ be a closed three-manifold, let $\lambda$ be a contact form on $Y$ with associated contact structure $\xi$, and let $\Gamma\in H_1(Y)$ be a class such that $c_1(\xi)+2\op{PD}(\Gamma)\in H^2(Y;\Z)$ is torsion\footnote{Such classes $\Gamma$ always exist because $TY$ is trivial.}. Then there exists a nonempty finite set of Reeb orbits whose total homology class is $\Gamma$.
\end{theorem}

Here $c_1(\xi)\in H^2(Y;\Z)$ denotes the first Chern class of the contact structure, regarded as a complex line bundle via the orientation given by $d\lambda$. Also $\op{PD}(\Gamma)\in H^2(Y;\Z)$ denotes the Poincar\'e dual of $\Gamma$, defined using the orientation on $Y$ given by $\lambda\wedge d\lambda$.

Prior to Taubes's result, various special cases of the Weinstein conjecture were proved, sometimes with stronger statements. For example, Hofer \cite{hofer93} proved that if $Y$ is a closed connected three-manifold, if $\lambda$ is a contact form on $Y$, and if (i) $Y\simeq S^3$ or (ii) $\pi_2(Y)\neq 0$ or (iii) the contact structure $\xi=\op{Ker}(\lambda)$ is overtwisted\footnote{A contact structure $\xi$ on a three-manifold $Y$ is {\bf overtwisted\/} if there exists an embedded disk $D\subset Y$ such that $\xi$ agrees with $TD$ on $\partial D$. The contact structure $\xi$ is {\bf tight\/} if it is not overtwisted. Eliashberg \cite{eliashberg} showed that on a closed oriented three-manifold, each homotopy class of oriented plane fields on $Y$ contains an overtwisted contact structure which is unique up to contact isotopy. On the other hand Etnyre-Honda \cite{etnyre-honda} showed that there exist closed orientable three-manifolds that do not admit any tight contact structure.}, then $\lambda$ has a {\em contractible\/} Reeb orbit. Note that not all contact forms on closed three-manifolds have contractible Reeb orbits; a counterexample is the canonical contact form on the unit cotangent bundle of ${\mathbb T}^2=\R^2/\Z^2$ with the standard metric. Abbas-Cieliebak-Hofer \cite{ach} showed that if $\lambda$ is a contact form on a closed three-manifold such that the associated contact structure $\xi=\op{Ker}(\lambda)$ is planar\footnote{A contact structure on a closed three-manifold $Y$ is {\bf planar\/} if it is supported in the sense of Giroux \cite{giroux} by an open book decomposition of $Y$ with genus zero pages.}, then there exists a nonempty finite set of Reeb orbits with total homology class zero. It remains an open question whether this is true without the hypothesis that $\xi$ is planar. Colin-Honda \cite{colin-honda} showed that for a contact form $\lambda$ on a closed three-manifold, if the associated contact structure is supported by an open book decomposition whose monodromy is pseudo-Anosov with sufficiently large fractional Dehn twist coefficient, then there are infinitely many simple Reeb orbits, in distinct free homotopy classes.

The results mentioned in the previous paragraph were all proved using pseudoholomorphic curve techniques. On the other hand, Taubes's proof of Theorem~\ref{thm:taubesweinstein} uses Seiberg-Witten theory, and this is currently the only known proof of the Weinstein conjecture in all cases in three dimensions. See \cite{tw} for a longer survey. In higher dimensions the Weinstein conjecture is largely an open question.

\subsection{Two Reeb orbits}

The three-dimensional case of the Weinstein conjecture asserts that every contact form on a closed three-manifold has at least one Reeb orbit, but various improvements of this result are possible.

\begin{theorem}[\cite{twoorbits}, see proof in \S\ref{sec:twoorbitsproof}]
\label{thm:twoorbits}
Let $\lambda$ be a contact form on a closed three-manifold $Y$. Then $\lambda$ has at least two simple Reeb orbits.
\end{theorem}

The lower bound of two simple Reeb orbits is the best possible without further hypotheses, as shown by the following example:

\begin{example}
\label{ex:ellipsoid}

Let $a,b>0$ be positive real numbers. Identify $\R^4=\C^2$ with coordinates $z_j=x_j+iy_j$ and define the {\bf ellipsoid\/}
\[
E(a,b) = \left\{z\in\C^2 \;\bigg|\;\frac{\pi|z_1|^2}{a} + \frac{\pi|z_2|^2}{b} \le 1\right\}.
\]
Let
\begin{equation}
\label{eqn:threedimensionalellipsoid}
Y = \partial E(a,b) = \left\{z\in\C^2 \;\bigg|\;\frac{\pi|z_1|^2}{a} + \frac{\pi|z_2|^2}{b} = 1 \right\}.
\end{equation}
Then $Y$ is a star-shaped hypersurface. For the contact form $\lambda=\lambda_0|_Y$, the Reeb vector field is given by
\begin{equation}
\label{eqn:ellipsoidReeb}
R = 2\pi\left(\frac{1}{a}\frac{\partial}{\partial\theta_1} + \frac{1}{b}\frac{\partial}{\partial\theta_2}\right)
\end{equation}
where $\theta_j$ denotes the argument of $z_j$, so that $\partial_{\theta_j}=x_j\partial_{y_j}-y_j\partial_{x_j}$. It follows from \eqref{eqn:ellipsoidReeb} that there is a simple Reeb orbit $\gamma_1$ of period $a$ given by the circle where $z_2=0$, and there is a simple Reeb orbit $\gamma_2$ of period $b$ given by the circle where $z_1=0$. If $a/b$ is irrational, then there are no other simple Reeb orbits. (On the other hand if $a/b$ is rational, then every point is on a Reeb orbit, and the simple Reeb orbits other than $\gamma_1$ and $\gamma_2$ have period equal to the least common multiple of $a$ and $b$.)
\end{example}

Note that $Y$ above, and more generally any star-shaped hypersurface in $\R^4$, is diffeomorphic to $S^3$. One can modify the above example to obtain contact forms on lens spaces with exactly two simple Reeb orbits. Recall that if $p,q$ are relatively prime positive integers, then the {\bf lens space\/} $L(p,q)$ is a smooth three-manifold defined as the quotient of the $\Z/p$ action on $S^3=\{z\in\C^2\mid \|z\|=1\}$ whose generator sends
\begin{equation}
\label{eqn:lensspaceaction}
(z_1,z_2) \longmapsto \left(e^{2\pi i/p}z_1,e^{2\pi iq/p}z_2\right).
\end{equation}
Equation \eqref{eqn:lensspaceaction} also describes a $\Z/p$ action on $\C^2$ which preserves the standard Liouville form \eqref{eqn:standardLiouvilleform} and the three-dimensional ellipsoid $Y$ in equation \eqref{eqn:threedimensionalellipsoid}. Thus the contact form $\lambda=\lambda_0|_Y$ descends to a contact form on the quotient three-manifold $Y/(\Z/p)$, which is diffeomorphic to the lens space $L(p,q)$, and this contact form has only two simple Reeb orbits when $a/b$ is irrational.

%%%%%%%%%%%%%%%%%%%%%%%%%%%%%%%%%%%%%

\subsection{Contact three-manifolds with exactly two simple Reeb orbits}

%%%%%%%%%%%%%%%%%%%%%%%%%%%%%%%%%%%%%

In fact, contact three-manifolds with exactly two simple Reeb orbits are quite rare, and mostly agree with the above examples. To start, we have:

\begin{theorem}[\cite{cghhl1} plus \cite{tw}]
\label{thm:exactlytwo}
Let $\lambda$ be a contact form on a closed three-manifold $Y$. Suppose that $\lambda$ has exactly two simple Reeb orbits. Then $Y$ is diffeomorphic to $S^3$ or a lens space, and the two simple Reeb orbits are the core circles of the solid tori in a genus one Heegaard splitting of $Y$.
\end{theorem}

\begin{remark}
The conclusion of Theorem~\ref{thm:exactlytwo} was proved in \cite{tw} under the additional hypothesis that the contact form $\lambda$ is nondegenerate (see Definition~\ref{def:nondegenerate} below). Then \cite{cghhl1} showed that this hypothesis is redundant: If a contact form on a closed three-manifold has exactly two simple Reeb orbits, then it must be nondegenerate!
\end{remark}

To clarify the above, for $t\in\R$, let $\phi_t:Y\to Y$ denote the time $t$ flow of the Reeb vector field $R$. Observe that the Lie derivative $\mathcal{L}_R\lambda=0$, so $\phi_t^*\lambda=\lambda$. Now let $\gamma:\R/T\Z\to Y$ be a Reeb orbit. Then the derivative of $\phi_T$ at $\gamma(0)$ restricts to a symplectic linear map
\[
P_\gamma: (\xi_{\gamma(0)},d\lambda) \longrightarrow (\xi_{\gamma(0)},d\lambda),
\]
which we call the {\bf linearized return map\/}. The eigenvalues of $P_\gamma$ are invariant under reparametrization of $\gamma$ by precomposition with a translation of $\R/T\Z$.

\begin{definition}
\label{def:nondegenerate}
The Reeb orbit $\gamma$ is {\bf nondegenerate\/} if the linearized return map $P_\gamma$ does not have $1$ as an eigenvalue. The contact form $\lambda$ is {\bf nondegenerate\/} if all (not necessarily simple) Reeb orbits are nondegenerate.
\end{definition}

For example, the ellipsoid $\partial E(a,b)$ is nondegenerate if and only if $a/b$ is irrational. For a given contact structure $\xi$, a $C^\infty$-generic contact form $\lambda$ with $\op{Ker}(\lambda)=\xi$ is nondegenerate; see e.g. \cite[\S6]{hwz98}.

\begin{remark}
More is proved in \cite{cghhl1} about contact forms with exactly two simple Reeb orbits. For example, in the case when $Y\simeq S^3$, it is shown that the two simple orbits have self-linking number $-1$ and form a Hopf link; and the symplectic actions and rotation numbers of the two simple orbits and the contact volume satisfy the same numerical relations as in an ellipsoid; see \cite[Eq.\ (1-3)]{cghhl1}. However there exist contact forms on $S^3$ with exactly two simple Reeb orbits whose Reeb flows are not conjugate to that of an ellipsoid. This can be seen from Katok's examples \cite{katok} or the construction in \cite{agz}; see \cite[Rem.\ 1.7]{cghhl1}.
\end{remark}

\subsection{Two or infinity}

It is not known whether there exists a contact form on a closed connected three-manifold with more than two but only finitely many simple Reeb orbits. In most cases, we know that this is not possible:

\begin{theorem}[\cite{cdr} plus \cite{cghhl2}]
\label{thm:toi}
Let $\lambda$ be a contact form on a closed connected three-manifold $Y$. Suppose that at least one of the following holds:
\begin{description}
\item{(a)} $\lambda$ is nondegenerate.
\item{(b)} $c_1(\xi)\in H^2(Y;\Z)$ is torsion.
\end{description}
Then $\lambda$ has either two or infinitely many simple Reeb orbits.
\end{theorem}

\begin{remark}
Hofer-Wysocki-Zehnder \cite{hwz03} showed that a star-shaped hypersurface in $\R^4$ has either two or infinitely many simple Reeb orbits, assuming that the contact form is nondegenerate and the stable and unstable manifolds of all hyperbolic Reeb orbits intersect transversely. The proof uses holomorphic curves to find a genus zero Birkhoff section for the Reeb flow, thus reducing the problem to the dynamics of area-preserving surface diffeomorphisms, and then invokes a theorem of Franks \cite{franks}, asserting that every area-preserving homeomorphism of an open disk has either one or infinitely many periodic points. In \cite{cghp}, this approach was generalized, using holomorphic curves arising from embedded contact homology, to prove two or infinitely many simple Reeb orbits for a contact form on a closed connected three-manifold satisfying both (a) and (b) above. In \cite{cdr} it was shown that one can assume only (a). In \cite{cghhl2} it was shown that one can assume only (b).
\end{remark}

%%%%%%%%%%%%%%%%%%%%%%%%%%%%%%

\subsection{Closing lemmas}
\label{sec:scl}

%%%%%%%%%%%%%%%%%%%%%%%%%%%%%%%

Irie \cite{irieclosing} proved that for a generic contact form on a closed three-manifold, not only are there infinitely many simple Reeb orbits, but the Reeb orbits are dense:

\begin{theorem}
[Irie \cite{irieclosing}]
\label{thm:iriedensity}
Let $Y$ be a closed three-manifold. For a $C^\infty$-generic contact form on $Y$, the union of the (images of the) Reeb orbits is dense in $Y$.
\end{theorem}

The key to the proof is a ``closing lemma'', asserting that given a contact form $\lambda$ on $Y$ and given a nonempty set $\mathcal{U}\subset Y$, one can make a $C^\infty$-small perturbation of $\lambda$ to create a Reeb orbit passing through $\mathcal{U}$. More precisely, Irie proved the following ``strong closing lemma'':

\begin{theorem}
[Irie \cite{irieclosing}, see proof in \S\ref{sec:irietrick}]
\label{thm:scl}
Let $Y$ be a closed three-manifold, let $\lambda$ be a contact form on $Y$, and let $f:Y\to \R_{\ge 0}$ be a smooth function, not identically zero. Then there exists $\tau\in[0,1]$ such that the contact form $e^{\tau f}\lambda$ has a Reeb orbit intersecting $\op{supp}(f)$.
\end{theorem}

This result should be compared with Pugh's $C^1$-closing lemma \cite{pugh}, which asserts that given a nonwandering point $x$ of a $C^1$-diffeomorphism $f$ of a closed manifold, one can find a $C^1$-small perturbation $g$ of $f$ so that $x$ is a periodic point of $g$. The proof requires great care in constructing the perturbation, and extending the result to higher regularity seems very difficult. By contrast, Irie's strong closing lemma asserts that one can create a Reeb orbit through $\mathcal{U}$ by a $C^\infty$ small perturbation of the contact form via any nonnegative perturbation supported in $\mathcal{U}$. As we will see in \S\ref{sec:irietrick}, the proof does not even require that $f$ is nonnegative, only that
\[
\int_Y \left(e^{2f} - 1\right) \lambda\wedge d\lambda \neq 0.
\]

\begin{remark}
It is not known whether Irie's strong closing lemma extends to all higher dimensional contact manifolds. This has been proved for the special case of ellipsoids and some related examples \cite{cdpt,cs,xue}. A difficulty with extending the methods presented in these lecture notes to general higher dimensional cases is that one does not expect sublinear growth of spectral invariants as in Theorem~\ref{thm:ckweyl} below. The reason is that in the formula for the Fredholm index of holomorphic curves in higher dimensions, see \S\ref{sec:fredholmindex}, the coefficient of the genus is nonpositive.
\end{remark}

\begin{remark}
Irie \cite{irieequi} also proved that for a $C^\infty$-generic contact form $\lambda$ on a closed three-manifold $Y$, an even stronger statement holds than the density of Reeb orbits in Theorem~\ref{thm:iriedensity}: namely, there is a sequence of Reeb orbits which is equidistributed with respect to the measure $\lambda\wedge d\lambda$.
\end{remark}

%%%%%%%%%%%%%%%%%%%%%%%%%%

\subsection{Quantitative closing lemmas}
\label{sec:qcl}

%%%%%%%%%%%%%%%%%%%%%%%%%

Irie's strong closing lemma asserts that given a contact form $\lambda$ on a closed three-manifold $Y$, one can create a Reeb orbit intersecting a nonempty open set $\mathcal{U}\subset Y$ by a $C^\infty$-small perturbation of $\lambda$ supported in $\mathcal{U}$. If $\lambda$ does not already have a Reeb orbit intersecting $\mathcal{U}$, and if the perturbation is small, then one might expect the Reeb orbit created to have a large period (because a short trajectory would not be perturbed enough to close up). One can then ask for a quantitative refinement of Irie's strong closing lemma: Given $L>0$, how much do we need to perturb the contact form in order to create a Reeb orbit through $\mathcal{U}$ of period at most $L$? As we will see, the answer is, roughly, $O(L^{-1})$. 

To make a precise statement, we first define an appropriate measure of the ``size'' of a deformation of the contact form (Definition~\ref{def:width} below).

If $n$ is a positive integer and $r>0$, define the ball
\[
B^{2n}(r) = \left\{z\in\C^n \;\big|\; \pi|z|^2\le r\right\},
\]
equipped with the standard symplectic form $\omega_0$ on $\C^n=\R^{2n}$. If $(X,\omega)$ is a $2n$-dimensional symplectic manifold, recall that its {\bf Gromov width\/} (also known as the {\bf ball capacity\/})
\[
c_{\op{Gr}}(X,\omega)\in(0,\infty]
\]
is defined to be the supremum of $r>0$ such that there exists a symplectic embedding
\[
B^{2n}(r)\hookrightarrow (X,\omega).
\]

\begin{definition}
Let $Y$ be a closed three-manifold, and let $\lambda$ be a contact form on $Y$. A {\bf positive deformation\/} of $\lambda$ is a smooth one-parameter family of contact forms
\begin{equation}
\label{eqn:posdef}
\{\lambda_\tau = e^{f_\tau}\lambda\}|_{\tau\in[0,1]}
\end{equation}
with the following properties:
\begin{itemize}
\item
$f_0\equiv 0$, so that $\lambda_0=\lambda$.
\item $f_1\ge 0$, and $f_1$ does not vanish identically.
\end{itemize}
We define the {\bf support\/} of the positive deformation \eqref{eqn:posdef} by
\begin{equation}
\label{eqn:support}
\op{supp}(\{\lambda_\tau\}) = \overline{\{y\in Y \mid \exists \tau\in[0,1] : f_\tau(y) \neq 0\}} \subset Y.
\end{equation}
\end{definition}

Recall that the {\bf symplectization\/} of $(Y,\lambda)$ is the four-manifold $\R\times Y$ with the symplectic form
\[
\omega = d(e^s\lambda),
\]
where $s$ denotes the $\R$ coordinate. Given a nonnegative function $f:Y\to\R_{\ge 0}$,  define
\[
M_{f} = \left\{(s,y)\in\R\times Y \;\big|\; 0 < s < f(y)\right\},
\]
equipped with the restriction of the symplectic form $d(e^s\lambda)$. In other words, $M_f$ is the region between $Y$ and the graph of $f$.

\begin{definition}
\label{def:width}
\cite{altspec}
Given a positive deformation $\left\{\lambda_\tau=e^{f_\tau}\lambda\right\}_{\tau\in[0,1]}$ as above, define its {\bf width\/} by
\[
\op{width}\left(\left\{\lambda_\tau\right\}_{\tau\in[0,1]}\right) = c_{\op{Gr}}\left(M_{f_1}\right) \in (0,\infty).
\]
\end{definition}

\begin{remark}
The width might at first seem like an unexpected way to measure the size of a deformation of the contact form, but it is what will work in the proof of the quantitative closing lemma below. Although the width may be difficult to compute, one can bound it as follows. Since symplectic embeddings respect volume, one has an upper bound
\[
c_{\op{Gr}}(M_f)^2 \le \int_Y\left(e^{2f}-1\right)\lambda\wedge d\lambda.
\]
And when $f:Y\to\R_{\ge 0}$ is not identically zero, one can show that there exists a constant $C>0$ such that
\[
c_{\op{Gr}}\left(M_{\tau f}\right) \ge C \tau
\]
for $\tau>0$ sufficiently small.
\end{remark}

We are now ready for the central definition that enters into quantitative closing lemmas:

\begin{definition}
\label{def:CloseL}
Let $Y$ be a closed three-manifold, let $\lambda$ be a contact form on $Y$, and let $L > 0$. Define\footnote{This is a slight tweak of the definition of $\op{Close}^L$ in \cite{altspec}. We would recover the definition in \cite{altspec} if we did not take the closure in \eqref{eqn:support}. The tweak seems necessary for our proof of Proposition~\ref{prop:sgcb} below to work.}
\[
\op{Close}^L(Y,\lambda) \in [0,\infty]
\]
to be the infimum of $\delta>0$ such that if $\{\lambda_\tau\}_{\tau\in[0,1]}$ is a positive deformation of $\lambda$, and if $\op{width}(\{\lambda_\tau\}) \ge \delta$, then there exists $\tau\in[0,1]$ such that the contact form $\lambda_\tau$ has a Reeb orbit $\gamma$ intersecting $\op{supp}(\{\lambda_\tau\})$ with action $\mathcal{A}(\gamma)\le L$.
\end{definition}

The idea of the definition is that if $\op{Close}^L(Y,\lambda)$ is finite, then this number tells us how much we need to deform the contact form in order to guarantee the creation of a Reeb orbit of action at most $L$. On the other hand, if $\op{Close}^L(Y,\lambda) = \infty$ then we do not know much. We will see in Theorem~\ref{thm:qcl} below that if $L$ is sufficiently large, then $\op{Close}^L(Y,\lambda)$ is finite, in fact $O(L^{-1})$.

Some more remarks to clarify Definition~\ref{def:CloseL}:

\begin{remark}
It follows from the definition that $\op{Close}^L(Y,\lambda)$ is a nonincreasing function of $L$.
\end{remark}

\begin{remark}
If $L$ is less than the minimum action of a Reeb orbit, then $\op{Close}^L(Y,\lambda) = +\infty$. The reason is that given $\delta>0$, one can consider the positive deformation given by $\lambda_\tau = e^{a\tau}\lambda$ where $a>0$ is a constant. One can show that if $a$ is sufficiently large, then this positive deformation will have width greater than $\delta$. Now the Reeb orbits of $e^{a\tau}\lambda$ are the same as the Reeb orbits of $\lambda$, except that they are reparametrized so that the action gets multiplied by $e^{a\tau}$. Thus this deformation does not create any Reeb orbits of action $\le L$.
\end{remark}

\begin{remark}
\label{rem:besse}
It is also possible that $\op{Close}^L(Y,\lambda)=0$. This is equivalent to the statement that every point in $Y$ is on (the image of) a Reeb orbit of action $\le L$. (This happens for example on $\partial E(a,b)$ when $a/b$ is rational and $L$ is the least common multiple of $a$ and $b$. See \cite{cgm,ff} for some characterizations of such contact forms.) The direction $(\Leftarrow)$ is immediate from the definition. To prove the reverse direction, if $\op{Close}^L(Y,\lambda)=0$, then for any nonempty open set $\mathcal{U}\subset Y$, one can produce a Reeb orbit intersecting $\mathcal{U}$ of action $\le L$ by an arbitrarily small positive deformation of $\lambda$ supported in $\mathcal{U}$. It then follows from the compactness of $Y$ that the original contact form $\lambda$ has a Reeb orbit intersecting the closure $\overline{\mathcal{U}}$ with action $\le L$. Since $\mathcal{U}$ was arbitrary, it follows that Reeb orbits of action $\le L$ are dense in $Y$. By another compactness argument, every point in $Y$ is on a Reeb orbit of action $\le L$.
\end{remark}

\begin{remark}
If $\lambda$ has no Reeb orbit of action $\le L$ intersecting $\mathcal{U}$, and if $\{\lambda_\tau\}$ is a positive deformation of $\lambda$ supported in $\mathcal{U}$ with width greater than $\op{Close}^L(Y,\lambda)$, then for some $\tau\in[0,1]$, the contact form $\lambda_\tau$ has a Reeb orbit $\gamma$ intersecting $\mathcal{U}$ with action $\le L$. One might expect that $\gamma$ is a concatenation of Reeb trajectories of $\lambda_\tau$ in $\mathcal{U}$ from $\partial\mathcal{U}$ to itself, together with Reeb trajectories of $\lambda$ in $Y\setminus\mathcal{U}$ from $\partial\mathcal{U}$ to itself. This may appear to be a contradiction if there is no Reeb trajectory of $\lambda$ in $Y\setminus\mathcal{U}$ from $\partial\mathcal{U}$ to itself taking time $\le L$. The resolution of the paradox is that in this case, the Reeb orbit $\gamma$ for $\lambda_\tau$ is contained entirely within $\mathcal{U}$. The idea is that the large deformation ``stirs up'' the Reeb flow in $\mathcal{U}$ enough to create such an orbit.
\end{remark}

\begin{remark}
The definition of $\op{Close}^L$ also makes sense for higher dimensional contact manifolds. This and more general notions are explored in \cite{ct23}.
\end{remark}

We can now state a general quantitative closing lemma:

\begin{theorem}[\cite{altspec}, see proof in \S\ref{sec:qclproofs}]
\label{thm:qcl}
Let $Y$ be a closed three-manifold and let $\lambda$ be a contact form on $Y$. Then
\[
\limsup_{L\to\infty}\left(L\cdot \op{Close}^L(Y,\lambda)\right) \le \op{vol}(Y,\lambda).
\]
\end{theorem}

Here $\op{vol}(Y,\lambda)$ denotes the {\bf contact volume\/} defined by
\[
\op{vol}(Y,\lambda) = \int_Y\lambda\wedge d\lambda > 0,
\]
where the integral is computed using the orientation on $Y$ given by $\lambda\wedge d\lambda$.

We can compute $\op{Close}^L$ exactly for the example of the ellipsoid $\partial E(a,b)$ when $L$ is large enough. The answer is given in terms of approximations of $a/b$ by rational numbers. To state the result, let $a,b$ be positive real numbers and let $L\ge \max(a,b)$. Let $m_-,n_-$ be relatively prime integers with $m_->0$ such that $n_-/m_-$ is maximized subject to the constraints $n_-/m_- \le a/b$ and $am_- \le L$. Similarly, let $m_+,n_+$ be relatively prime integers with $n_+>0$ such that $m_+/n_+$ is maximized\footnote{With this definition of $m_+$ and $n_+$, we are making a slight correction to the statement of \cite[Thm.\ 1.10]{altspec}. The proof is unchanged.} subject to the constraints $m_+/n_+\le b/a$ and $bn_+\le L$.

\begin{theorem}[\cite{altspec}, see partial proof in \S\ref{sec:qclproofs}]
\label{thm:ecl}
If $a,b>0$ and $L\ge \max(a,b)$, then with the above notation we have
\[
\op{Close}^L(\partial E(a,b)) = \min(am_--bn_-, bn_+ - am_+).
\]
\end{theorem}

\begin{remark}
If $a$ and $b$ are relatively prime positive integers and $L \ge ab$, then we can take $n_-=n_+=a$ and $m_-=m_+=b$, and we recover the fact that $\op{Close}^L(\partial E(a,b)) = 0$ by Remark~\ref{rem:besse}.
\end{remark}

\begin{remark}
The combination of Theorems~\ref{thm:qcl} and \ref{thm:ecl} implies that an irrational number $a/b$ has a sequence of rational approximations $n/m$ with error on the order of $1/m^2$, similarly to Dirichlet's approximation theorem. If $a/b$ has irrationality exponent greater than $2$, then $\liminf_{L\to\infty}L\cdot \op{Close}^L(\partial E(a,b))=0$.
\end{remark}

\subsection{Outline of the rest of these lecture notes}

In \S\ref{sec:spec}, we introduce some key properties of the ``elementary spectral invariants'' of contact three-manifolds and use these to prove Theorems~\ref{thm:twoorbits},  \ref{thm:scl}, \ref{thm:qcl}, and \ref{thm:ecl}. To prepare for the definition of the elementary spectral invariants, in \S\ref{sec:alt} we digress to review the ``alternative ECH capacities'' of symplectic four-manifolds. In \S\ref{sec:comp}, we review the ECH index and use this to compute some basic examples of alternative ECH capacities. Finally, in \S\ref{sec:defspec} we define the elementary spectral invariants and prove their basic properties.

\subsection{Additional topics}

The following are some related topics which are beyond the scope of these lecture notes.

\paragraph{Reeb orbits frequently intersecting a symplectic surface.} The closing lemmas in \S\ref{sec:scl} and \S\ref{sec:qcl} involve creating a Reeb orbit by a small perturbation of the contact form. Similar methods are used in \cite{calabi2} to show that under certain hypotheses, given a surface $\Sigma$ in a contact three-manifold $(Y,\lambda)$, there exist Reeb orbits $\gamma$ with an explicit lower bound on the intersection number $\gamma\cdot \Sigma$ divided by the action $\mathcal{A}(\gamma)$. These Reeb orbits exist without any perturbation of the contact form.

\paragraph{ECH spectral invariants.} Some of the theorems stated above, namely Theorems~\ref{thm:taubesweinstein}, \ref{thm:exactlytwo}, and \ref{thm:toi}, cannot (currently) be proved using the methods introduced here. The first of these theorems can be deduced from, and is a step towards, Taubes's isomorphism \cite{taubes} between Seiberg-Witten-Floer theory \cite{km} and embedded contact homology \cite{bn}. The other two theorems use the ECH spectral invariants from \cite{qech}, which contain more information than the elementary spectral invariants introduced here.

\paragraph{Area-preserving surface diffeomorphisms.} There is a similar story for area-preserving diffeomorphisms of a closed surface. In this context, a version of Irie's strong closing lemma was proved in \cite{danclosing}, and also in \cite{pfh4} with the help of \cite{ucyclic} by a different approach which also yields quantitative closing lemmas. Both approaches make use of spectral invariants in periodic Floer homology, an analogue of embedded contact homology for area-preserving surface diffeomorphisms \cite{pfh3}. There is also an ``elementary'' alternative to PFH spectral invariants \cite{edtmairelementary} which can recover some of the same applications. A version of Irie's generic equidistribution theorem for area-preserving surface diffeomorphisms was proved in \cite{prasad}, and generalized to surfaces with boundary in \cite{pp}.

\paragraph{Minimal hypersurfaces.} There is also a remarkable parallel story about minimal hypersurfaces in Riemannian manifolds. For a generic metric, it was shown in \cite{imn} that minimal hypersurfaces are dense, and in \cite{mns} that minimal hypersurfaces are equidistributed. The proofs use a ``Weyl law'' for the ``volume spectrum'' of Riemannian manifolds proved in \cite{lmn}, which is formally similar to the Weyl law for elementary spectral invariants of contact three-manifolds in Theorem~\ref{thm:ckweyl} below.

%%%%%%%%%%%%%%%%%%%%%%%%%%%%%%%%%

\section{Elementary spectral invariants: properties and applications}
\label{sec:spec}

%%%%%%%%%%%%%%%%%%%%%%%%%%%%%%%%%%

In \cite{qech}, embedded contact homology was used to define the ECH spectral invariants of contact three-manifolds, which can be used to prove the dynamical theorems stated in \S\ref{sec:intro}. The ``elementary spectral invariants'' of contact three-manifolds, defined in \cite{altspec}, are a simplified version of the ECH spectral invariants, which behave similarly to the original ECH spectral invariants, but which require much less technology to establish. In this section we introduce some essential properties of the elementary spectral invariants, deferring the construction of the invariants and the proofs of their properties to \S\ref{sec:defspec}. The key nontrivial property is a ``Weyl law'' (Theorem~\ref{thm:ckweyl}). We then explain how these invariants can be used to prove some of the theorems stated in \S\ref{sec:intro}, namely Irie's strong closing lemma (Theorem~\ref{thm:scl}), its refinement in the general quantitative closing lemma (Theorem~\ref{thm:qcl}), the quantitative closing lemma for an ellipsoid (Theorem~\ref{thm:ecl}), and the existence of two Reeb orbits (Theorem~\ref{thm:twoorbits}). The main technique used here is to investigate the change in the spectral invariants under suitable deformations of the contact form, and to then draw dynamical conclusions from this.

%%%%%%%%%%%%%%%%%%%%%%%%%%%%%%%%%%%%

\subsection{Properties of elementary spectral invariants}
\label{sec:propspec}

%%%%%%%%%%%%%%%%%%%%%%%%%%%%%%%%%%%

If $Y$ is a closed three-manifold and $\lambda$ is a contact form on $Y$, its elementary spectral invariants consist of real numbers
\[
c_k(Y,\lambda) \in [0,+\infty]
\]
for each nonnegative integer $k$. Roughly speaking, $c_k(Y,\lambda)$ is the max-min energy of holomorphic curves in $\R\times Y$ constrained to pass through $k$ points. We will give the precise definition in \S\ref{subsec:defspec}. For now, we state some basic properties of these invariants which follow from the definition and standard properties of holomorphic curves.

\begin{definition}
In a contact manifold $(Y,\lambda)$, an {\bf orbit set\/} is a finite set of pairs $\alpha=\{(\alpha_i,m_i)\}$, where the $\alpha_i$ are distinct simple Reeb orbits, and the $m_i$ are positive integers. The {\bf symplectic action\/} of the orbit set $\alpha$ is defined by
\[
\mathcal{A}(\alpha) = \sum_i m_i \mathcal{A}(\alpha_i).
\]
Note that $\mathcal{A}(\alpha)\ge 0$, with equality if and only if $\alpha=\emptyset$.
\end{definition}

\begin{proposition}[\cite{altspec}, see proof in \S\ref{sec:ckbasicproofs}]
\label{prop:ckproperties}
The elementary spectral invariants $c_k$ satisfy the following properties:
\begin{description}
\item{(Increasing)} We have
\[
0 = c_0(Y,\lambda) < c_1(Y,\lambda) \le c_2(Y,\lambda) \le \cdots \le +\infty.
\]
\item{(Conformality)} If $r>0$ is a constant then $c_k(Y,r\lambda) = r\cdot c_k(Y,\lambda)$.
\item{(Monotonicity)} If $f:Y\to\R_{\ge 0}$ is a smooth function, then
\[
c_k(Y,\lambda) \le c_k(Y,e^f\lambda).
\]
\item{($C^0$-Continuity)} For fixed $\lambda$ and $k$, the function $C^\infty(Y;\R)\to[0,\infty]$ sending $f\mapsto c_k(Y,e^f\lambda)$ is continuous with respect to the $C^0$ topology on $C^\infty(Y;\R)$.
\item{(Spectrality)} If $c_k(Y,\lambda)<\infty$, then there exists an orbit set $\alpha$ with $c_k(Y,\lambda) = \mathcal{A}(\alpha)$.
\item{(Disjoint Union)} If $(Y_i,\lambda_i)$ are closed contact three-manifolds for $i=1,\ldots,m$, then
\[
c_k\left(\coprod_{i=1}^m(Y_i,\lambda_i)\right) = \max_{k_1+\cdots+k_m=k}\sum_{i=1}^m c_{k_i}(Y_i,\lambda_i).
\]
\item{(Width Bound)} If $f:Y\to\R_{\ge 0}$ is a smooth function, then
\[
c_{k+1}(Y,e^{f}\lambda) \ge c_k(Y,\lambda) + c_{\op{Gr}}(M_f).
\]
\end{description}
\end{proposition}

\begin{remark}
We will see from the Weyl law in Theorem~\ref{thm:ckweyl} below that in fact $c_k(Y,\lambda)<\infty$ always holds. This is a nontrivial fact, and for general three-manifolds its proof currently requires Seiberg-Witten theory, although we will give an easier proof for the special case of star-shaped hypersurfaces in Lemma~\ref{lem:enhancedvariant} below. An independent proof that $c_1(Y,\lambda)<\infty$ would give a new proof of the Weinstein conjecture in three dimensions (see Theorem~\ref{thm:weinstein} below).
\end{remark}

\begin{remark}
One can define similar elementary spectral invariants for higher dimensional contact manifolds, which satisfy all the properties in Proposition~\ref{prop:ckproperties}; see \cite{ct23,irieicm} and Remark~\ref{rem:ckalthd} below. However for higher dimensional contact manifolds we do not know whether such invariants are always finite. (If yes, then this would imply the higher dimensional Weinstein conjecture.)
\end{remark}

Here is a simple example where we can compute the elementary spectral invariants. If $a,b>0$ are positive real numbers, define $(N_k(a,b))_{k\ge 0}$ to be the sequence of all linear combinations $am+bn$ with $m,n\in\Z_{\ge 0}$, listed in nondecreasing order with repetitions. For example,
\[
(N_k(2,3))_{k\ge 0} = (0,2,3,4,5,6,6,7,8,8,9,\ldots).
\]

\begin{proposition}[\cite{altspec}, see proof in \S\ref{sec:ckmoreproofs}]
\label{prop:ckellipsoid}
The elementary spectral invariants of the three-dimensional ellipsoid $\partial E(a,b)$ are given by
\[
c_k(\partial E(a,b)) = N_k(a,b).
\]
\end{proposition}

Specializing to the case $a=b$, we get:

\begin{corollary}
\label{cor:ckball}
The elementary spectral invariants of the three-sphere $\partial B^4(a)$ are given by
\[
c_k(\partial B^4(a)) = da,
\]
where $d$ is the unique nonnegative integer such that
\begin{equation}
\label{eqn:uniqueinteger}
d^2+d \le 2k \le d^2+3d.
\end{equation}
\end{corollary}

More generally, there are combinatorial formulas for the elementary spectral invariants of the boundaries in $\R^4$ of ``convex toric domains'' (see Theorem~\ref{thm:ckconvex} below) and ``concave toric domains'' (see \cite[Rem.\ 1.16]{altspec}).

The key nontrivial property of the elementary spectral invariants $c_k$, which is essential to dynamical applications, is the following ``Weyl law'':

\begin{theorem}
[\cite{altspec}]
\label{thm:ckweyl}
If $Y$ is a closed three-manifold and $\lambda$ is a contact form on $Y$, then
\begin{equation}
\label{eqn:ckweyl}
\boxed{
\lim_{k\to\infty} \frac{c_k(Y,\lambda)^2}{k} = 2\op{vol}(Y,\lambda).
}
\end{equation}
\end{theorem}

\begin{remark}
In \S\ref{sec:ssweyl} we will prove the Weyl law in the special case when $Y$ is a star-shaped hypersurface, using basic properties of elementary spectral invariants. For a general three-manifold $Y$, the proof of the Weyl law is beyond the scope of these lecture notes. The general proof uses a basic relation between elementary spectral invariants and ECH spectral invariants proved in \cite[Thm.\ 6.1]{altspec}, together with a Weyl law for ECH spectral invariants proved in \cite{vc}. The proof of the latter Weyl law uses Taubes's isomorphism \cite{taubes} between embedded contact homology \cite{bn} and Seiberg-Witten Floer homology, together with properties of the latter proved by Kronheimer-Mrowka \cite{km}.

See \cite{cgweyl} for a survey of some related Weyl laws and their applications.
\end{remark}

\begin{example}
\label{ex:ellipsoidweyl}
Let us verify the Weyl law for the example of the ellipsoid $Y = \partial E(a,b)$. Recall from Proposition~\ref{prop:ckellipsoid} that $c_k(\partial E(a,b))=N_k(a,b)$. Now $N_k(a,b)$ has the following useful description in terms of lattice point geometry. Given $L > 0$, let $T(L)$ denote the triangle in the first quadrant of the plane whose edges are segments of the $x$ axis, the $y$ axis, and the line $ax+by=L$. Then for $k>0$, the number $N_k(a,b)$ is the minimum of $L>0$ such that the triangle $T(L)$ encloses at least $k+1$ lattice points (including lattice points on the boundary). This triangle has area
\[
\op{Area}(T(L)) = \frac{L^2}{2ab}.
\]
And for $k$ large, we have
\[
\op{Area}(T(L)) = k + O(k^{1/2}).
\]
Thus
\[
\lim_{k\to\infty}\frac{c_k(\partial E(a,b))^2}{k} = \lim_{L\to \infty}\frac{L^2}{\op{Area}(T(L))} = 2ab.
\]
This confirms the Weyl law since\footnote{To do this volume calculation, note that by Stokes' theorem, the contact volume is given by $\op{vol}(\partial E(a,b)) = 2\op{vol}(E(a,b))$, where on the right hand side we use the Euclidean volume. And for any domain $\Omega\subset \R^2_{\ge 0}$, we have $\op{vol}\left(\left\{z\in\C^2\mid (\pi|z_1|^2,\pi|z_2|^2)\in\Omega\right\}\right) = \op{Area}(\Omega)$.} $\op{vol}(\partial E(a,b)) = ab$.
\end{example}

As a first illustration of the power of the Weyl law, we can now prove:

\begin{theorem}
[three dimensional Weinstein conjecture]
\label{thm:weinstein}
Let $Y$ be a closed three-manifold and let $\lambda$ be a contact form on $Y$. Then $\lambda$ has a Reeb orbit.
\end{theorem}

\begin{proof}
By the Increasing property in Proposition~\ref{prop:ckproperties} and the Weyl law in Theorem~\ref{thm:ckweyl}, we have $0 < c_1(Y,\lambda)<\infty$. By the Spectrality property in Proposition~\ref{prop:ckproperties}, there exists an orbit set $\alpha$ with $\mathcal{A}(\alpha)= c_1(Y,\lambda)$. Since $\mathcal{A}(\alpha) > 0$, it follows that $\alpha$ is not the empty set.
\end{proof}

%%%%%%%%%%%%%%%%%%%%%%%%%%%%%%%%%%%%%%%%%%

\subsection{Proof of Irie's strong closing lemma}
\label{sec:irietrick}

%%%%%%%%%%%%%%%%%%%%%%%%%%%%%%%%%%%%%%%%%%

We now use the elementary spectral invariants to prove Irie's strong closing lemma (Theorem~\ref{thm:scl}), following\footnote{The proof in \cite{irieclosing} used the original ECH spectral invariants, but the proof here using the elementary spectral invariants goes by the same argument.} \cite{irieclosing}. 

\begin{definition}
If $(Y,\lambda)$ is a closed contact manifold, the {\bf action spectrum\/}\footnote{In other contexts in the literature, the term ``action spectrum'' more commonly refers to the set of actions of individual orbits, not orbit sets.} of $(Y,\lambda)$ is the set
\[
S(Y,\lambda) = \{\mathcal{A}(\alpha) \mid \mbox{$\alpha$ is an orbit set in $(Y,\lambda)$}\} \subset \R_{\ge 0}.
\]
\end{definition}

We will need the following lemma:

\begin{lemma}
\label{lem:measurezero}
(\cite[Lem.\ 2.2]{irieclosing}, cf. \cite[Lem.\ 3.8]{schwarz})
Let $(Y,\lambda)$ be a closed contact manifold. Then the action spectrum $S(Y,\lambda)$ has measure zero in $\R$.
\end{lemma}

\begin{proof}[Proof of Theorem~\ref{thm:scl}]
Let $Y$ be a closed three-manifold, let $\lambda$ be a contact form on $Y$, and let $f:Y\to\R_{\ge 0}$ be a smooth function, not identically zero. We need to show that there exists $\tau\in[0,1]$ such that the contact form $e^{\tau f}\lambda$ has a Reeb orbit intersecting $\op{supp}(f)$.

Suppose to get a contradiction that no such $\tau$ exists. Then for all $\tau\in[0,1]$, the Reeb orbits of $e^{\tau f}\lambda$ map to $Y\setminus\op{supp}(f)$, where $e^{\tau f}\lambda$ agrees with $\lambda$. Thus for all $\tau\in[0,1]$, the Reeb orbits of $e^{\tau f}\lambda$ and their symplectic actions are the same as those of $\lambda$. Consequently, the action spectrum
\begin{equation}
\label{eqn:sas}
S\left(Y,e^{\tau f}\lambda\right) = S(Y,\lambda)
\end{equation}
for all $\tau \in [0,1]$.

For fixed $k\in\Z_{\ge 0}$, consider the function
\[
\begin{split}
[0,1] &\longrightarrow [0,\infty],\\
\tau & \longmapsto c_k\left(Y,e^{\tau f}\lambda\right).
\end{split}
\]
By the Increasing property in Proposition~\ref{prop:ckproperties} and the Weyl law in Theorem~\ref{thm:ckweyl}, this function takes values in $[0,\infty)$. By the Spectrality property in Proposition~\ref{prop:ckproperties} and equation \eqref{eqn:sas}, this function takes values in the set $S(Y,\lambda)$, which has measure zero by Lemma~\ref{lem:measurezero}. By the $C^0$-Continuity property in Proposition~\ref{prop:ckproperties}, this function is continuous. Since it takes values in a measure zero set, it must be constant. We conclude that
\begin{equation}
\label{eqn:ckf}
c_k(Y,\lambda) = c_k\left(Y,e^f\lambda\right)
\end{equation}
for all $k$.

It follows from \eqref{eqn:ckf} and the Weyl law applied to both $\lambda$ and $e^f\lambda$ that
\[
\op{vol}(Y,\lambda) = \op{vol}\left(Y,e^f\lambda\right).
\]
On the other hand, we compute that
\[
\op{vol}\left(Y,e^f\lambda\right) = \int_Y e^{2f}\lambda\wedge d\lambda > \op{vol}(Y,\lambda)
\]
since $f:Y\to \R_{\ge 0}$ is not identically zero. This gives the desired contradiction.
\end{proof}

%%%%%%%%%%%%%%%%%%%%%%%%%%%%%%%%%%%%%%%%%%%

\subsection{The spectral gap closing bound}

%%%%%%%%%%%%%%%%%%%%%%%%%%%%%%%%%%%%%%%%%

We now present a modification of Irie's argument from \S\ref{sec:irietrick} which allows one to bound $\op{Close}^L(Y,\lambda)$ and thus prove quantitative closing lemmas. The argument below is an adaptation of an argument in \cite{pfh4} for area-preserving surface diffeomorphisms.

\begin{definition}
Let $Y$ be a closed three-manifold, let $\lambda$ be a contact form on $Y$, and let $L>0$. Define the {\bf minimum spectral gap\/}
\[
\op{Gap}^L(Y,\lambda) = \inf\{c_{k+1}(Y,\lambda) - c_{k}(Y,\lambda) \mid  c_{k+1}(Y,\lambda) \le L\} \in [0,+\infty].
\]
\end{definition}

\begin{proposition}
\label{prop:sgcb}
\cite{altspec}
Let $Y$ be a closed three-manifold, let $\lambda$ be a contact form on $Y$, and let $L>0$. Then
\[
\op{Close}^L(Y,\lambda) \le \op{Gap}^L(Y,\lambda).
\]
\end{proposition}

\begin{proof}
Let $\{\lambda_\tau = e^{f_\tau}\lambda\}_{\tau\in[0,1]}$ be a positive deformation of $\lambda$, and suppose that $\op{width}(\{\lambda_\tau\}) > \op{Gap}^L(Y,\lambda)$. This last condition means that there exists a nonnegative integer $k$ such that
\begin{align}
\label{eqn:sgcb1}
c_{k+1}(Y,\lambda) & \le L, \\
\label{eqn:sgcb2}
c_{k+1}(Y,\lambda) - c_k(Y,\lambda) &< \op{width}(\{\lambda_\tau\}).
\end{align}
We need to show that there exists $\tau\in[0,1]$ such that the contact form $\lambda_\tau$ has a Reeb orbit $\gamma$ intersecting $\op{supp}(\{\lambda_\tau\})$ with symplectic action $\mathcal{A}(\gamma) \le L$. By a compactness argument, it is enough to show that for every $\epsilon>0$, there exists $\tau\in[0,1]$ such that the contact form $\lambda_\tau$ has a Reeb orbit $\gamma$ intersecting $\op{supp}(\{\lambda_\tau\})$ with symplectic action $\mathcal{A}(\gamma) \le L+\epsilon$.

Suppose to get a contradiction that there exists $\epsilon>0$ such that no such $\tau$ exists. Then for each $\tau\in[0,1]$, the Reeb orbits of $\lambda_\tau$ with action $\le L+\epsilon$ are in the complement of $\op{supp}(\{\lambda_\tau\})$, and so these orbits, and their symplectic actions, are the same as those of $\lambda$. Thus the action spectra satisfy
\begin{equation}
\label{eqn:ttass}
S(\lambda_\tau) \cap [0,L+\epsilon] = S(\lambda)\cap [0,L+\epsilon].
\end{equation}
By the $C^0$-continuity and Spectrality properties in Proposition~\ref{prop:ckproperties},
\[
c_{k+1}(Y,\lambda_\tau)\in S(\lambda_\tau)
\]
is a continuous function of $\tau\in[0,1]$. Since this function has value $\le L$ at $\tau=0$ by \eqref{eqn:sgcb1}, and since the set $S(\lambda)\cap [0,L+\epsilon]$ has measure zero by Lemma~\ref{lem:measurezero}, it follows from \eqref{eqn:ttass} that $c_{k+1}(Y,\lambda_\tau)$ is a constant function of $\tau\in[0,1]$. Thus
\begin{equation}
\label{eqn:sgcb3}
c_{k+1}(Y,\lambda_1) = c_{k+1}(Y,\lambda).
\end{equation}
By the Width Bound property in Proposition~\ref{prop:ckproperties}, we have
\begin{equation}
\label{eqn:sgcb4}
c_{k+1}(Y,\lambda_1) - c_k(Y,\lambda) \ge \op{width}(\{\lambda_\tau\}).
\end{equation}
Combining \eqref{eqn:sgcb2}, \eqref{eqn:sgcb3}, and \eqref{eqn:sgcb4} gives the desired contradiction.
\end{proof}

%%%%%%%%%%%%%%%%%%%%%%%%%%%%%%%%%%%%%%%%%%%

\subsection{Proofs of quantitative closing lemmas}
\label{sec:qclproofs}

%%%%%%%%%%%%%%%%%%%%%%%%%%%%%%%%%%%%%%%%%%%%

We now outline how Proposition~\ref{prop:sgcb} can be used to prove some quantitative closing lemmas.

\begin{proof}[Proof of Theorem~\ref{thm:qcl}.]
By Proposition~\ref{prop:sgcb}, it is enough to show that
\begin{equation}
\label{eqn:limsupgap}
\limsup_{L\to\infty}\left(L\cdot \op{Gap}^L(Y,\lambda)\right) \le \op{vol}(Y,\lambda).
\end{equation}
We need an elementary lemma, which is that if $(c_k)_{k \ge 0}$ is a nondecreasing sequence with $\lim_{k\to\infty} c_k^2/k=2V$, then
\begin{equation}
\label{eqn:elementarylemma}
\limsup_{L\to\infty} \left(L \cdot \inf_{c_{k+1}\le L}(c_{k+1}-c_k)\right) \le V;
\end{equation}
see \cite[\S6.4]{altspec} for a proof\footnote{Some intuition is that if $c_k^2/k=2V$ held for each $k$, then \eqref{eqn:elementarylemma} would follow immediately, with $\limsup$ replaced by $\lim$. In general we can expect some spectral gaps to be smaller.}. Combining \eqref{eqn:elementarylemma} with the Weyl law \eqref{eqn:ckweyl} proves \eqref{eqn:limsupgap}.
\end{proof}

\begin{proof}[Proof of Theorem~\ref{thm:ecl} (partial).]
We first prove the inequality
\[
\op{Close}^L(\partial E(a,b)) \le \min(am_--bn_-, bn_+ - am_+).
\]
By Proposition~\ref{prop:sgcb}, it is enough to show that
\begin{equation}
\label{eqn:ellipsoidgapbound}
\op{Gap}^L(\partial E(a,b)) \le \min(am_--bn_-, bn_+ - am_+).
\end{equation}
By Proposition~\ref{prop:ckellipsoid}, each of the numbers $am_-, bn_-, bn_+, am_+$ is a spectral invariant $c_k(\partial E(a,b))$ for some $k$. Recall that in the notation of Theorem~\ref{thm:ecl} we have $bn_- \le am_- \le L$ and $am_+\le bn_+ \le L$. Thus each of these spectral invariants is $\le L$. The inequality \eqref{eqn:ellipsoidgapbound} then follows from the definition of $\op{Gap}^L$.

To prove the reverse inequality 
\begin{equation}
\label{eqn:revineq}
\op{Close}^L(\partial E(a,b)) \ge \min(am_--bn_-, bn_+ - am_+),
\end{equation}
let $\mathcal{U}$ denote the complement in $\partial E(a,b)$ of the two simple Reeb orbits $\gamma_1$ and $\gamma_2$ (see Example~\ref{ex:ellipsoid}). It is enough to show that if $\delta$ is less than the right hand side of \eqref{eqn:revineq}, then there exists a positive deformation $\{\lambda_\tau\}$ of $\lambda_0$, whose width is at least $\delta$, with support in $\mathcal{U}$, such that for all $\tau\in[0,1]$, the contact form $\lambda_\tau$ does not have a Reeb orbit of action $\le L$ intersecting $\mathcal{U}$. Such a deformation is constructed using toric geometry in \cite[\S5.3]{altspec}.
\end{proof}

%%%%%%%%%%%%%%%%%%%%%%%%%%%%%%%%%%%%%%%%%%%%%%%%

\subsection{Proof of the existence of two simple Reeb orbits}
\label{sec:twoorbitsproof}

We now use the spectral gap closing bound (Proposition~\ref{prop:sgcb}) and the Weyl law (Theorem~\ref{thm:ckweyl}) to prove the existence of two simple Reeb orbits. This is a simplification, suggested by \cite{ff}, of the original proof in \cite{twoorbits}.

\begin{proof}[Proof of Theorem~\ref{thm:twoorbits}.]
Let $Y$ be a closed three-manifold and let $\lambda$ be a contact form on $Y$. We need to show that $\lambda$ has at least two simple Reeb orbits.

By the Weinstein conjecture (Theorem~\ref{thm:weinstein}), we know that there is at least one simple Reeb orbit. Suppose to get a contradiction that there is only one simple Reeb orbit $\gamma$, and let $T$ denote its symplectic action.

By the Spectrality property in Proposition~\ref{prop:ckproperties} (and the finiteness of spectral invariants as in Theorem~\ref{thm:weinstein}), for each nonnegative integer $k$ we have
\[
c_k(Y,\lambda) = a_k T
\]
for some nonnegative integer $a_k$. By the Increasing property in Proposition~\ref{prop:ckproperties}, we have
\begin{equation}
\label{eqn:tiis}
a_k \le a_{k+1}.
\end{equation}
In fact the above inequality must be strict. Otherwise $\op{Gap}^L(Y,\lambda)=0$ when $L\ge c_k(Y,\lambda)$, so Proposition~\ref{prop:sgcb} gives $\op{Close}^L(Y,\lambda)=0$. Then Remark~\ref{rem:besse} implies that every point in $Y$ is on a Reeb orbit, contradicting the fact that there is only one simple Reeb orbit.

Since the inequality \eqref{eqn:tiis} is strict, it follows that
\[
c_k(Y,\lambda) \ge k T.
\]
Therefore
\[
\lim_{k\to\infty}\frac{c_k(Y,\lambda)^2}{k} \ge \lim_{k\to\infty} kT^2 = +\infty,
\]
which contradicts the Weyl law \eqref{eqn:ckweyl}.
\end{proof}

%%%%%%%%%%%%%%%%%%%%%%%%%%%%%%%%%%%%%%%%%%%%%%%%%%%%%%%%%%%%%

\section{Alternative ECH capacities}
\label{sec:alt}

%%%%%%%%%%%%%%%%%%%%%%%%%%%%%%%%%%%%%%%%%%%%%%%

We now digress from dynamics to review the ``alternative ECH capacities'' of four-dimensional symplectic manifolds defined in \cite{altech}. These are numerical invariants which give obstructions to four-dimensional symplectic embeddings, and they are a simplification of the original ECH capacities defined in \cite{qech}. The key idea, inspired by \cite{ms}, is to define capacities as a max-min of energy of certain holomorphic curves satisfying suitable constraints. This digression will prepare us to define the elementary spectral invariants $c_k$ of a contact three-manifold $(Y,\lambda)$ by a variant of this construction in \S\ref{sec:defspec}, where the symplectic four-manifold is a subset of the symplectization $(\R\times Y,d(e^s\lambda))$.

%%%%%%%%%%%%%%%%%%%%%%%%%%%%%%%%%%%%%%%%%%%%

\subsection{Definition of alternative ECH capacities}
\label{sec:defaltech}

%%%%%%%%%%%%%%%%%%%%%%%%%%%%%%%%%%%%%%%%%%%%%%%%%%%%%%%%%

Let $(X,\omega)$ be a symplectic four-manifold. The alternative ECH capacities of $(X,\omega)$ are a sequence of (extended) real numbers
\begin{equation}
\label{eqn:ckaltincreasing}
0 = c_0^{\op{Alt}}(X,\omega) < c_1^{\op{Alt}}(X,\omega) \le c_2^{\op{Alt}}(X,\omega) \le \cdots \le +\infty.
\end{equation}

To define these, we need the following preliminaries.

\begin{definition}
\label{def:lambdacompatible}
Let $Y$ be a three-manifold, and let $\lambda$ be a contact form on $Y$ with contact structure $\xi=\op{Ker}(\lambda)$. An almost complex structure $J$ on $\R\times Y$ is {\bf $\lambda$-compatible\/} if:
\begin{itemize}
\item $J(\partial_s)=R$, where $s$ denotes the $\R$ coordinate and $R$ denotes the Reeb vector field on $Y$.
\item $J(\xi)=\xi$, rotating positively with respect to $d\lambda$, i.e.\ $d\lambda(v,Jv)\ge 0$ for $v\in\xi$.
\item $J$ is invariant under translation of the $\R$ factor in $\R\times Y$.
\end{itemize}
We denote the set of $\lambda$-compatible almost complex structures on $\R\times Y$ by $\mathcal{J}(Y,\lambda)$.
\end{definition}

\begin{definition}
A four-dimensional {\bf Liouville domain\/} is a pair $(X,\lambda)$, where $X$ is a compact four-manifold with boundary $Y$, and $\lambda$ is a $1$-form on $X$, such that:
\begin{itemize}
\item
$d\lambda=\omega$ is a symplectic form on $X$,
\item
$\lambda|_Y$ is a contact form on $Y$, and
\item
the orientation of $Y$ given by $\lambda\wedge d\lambda$ agrees with the boundary orientation for the orientation of $X$ given by $\omega\wedge\omega$.
\end{itemize}
The Liouville domain $(X,\lambda)$ is {\bf nondegenerate\/} if the contact form $\lambda|_Y$ is nondegenerate as in Definition~\ref{def:nondegenerate}. We say that $\lambda$ is a {\bf Liouville form\/} for $(X,\omega)$.
\end{definition}

For example, if $Y\subset\R^4$ is a star-shaped hypersurface, and if $X\subset\R^4$ is the compact region that it bounds, then $(X,\lambda)$ is a Liouville domain, where $\lambda$ is the restriction to $X$ of the standard Liouville form \eqref{eqn:standardLiouvilleform}.

Given a Liouville domain $(X,\lambda)$ as above, there is a unique vector field $V$ on $X$ such that $\imath_V\omega=\lambda$. The vector field $V$ is transverse to $Y$, pointing out of $X$. Given $\epsilon>0$, the flow of $V$ determines a neighborhood $N$ of $Y$ in $X$ and an identification
\[
N \simeq (-\epsilon,0]\times Y,
\]
identifying $\lambda$ on the left hand side with $e^s\lambda|_Y$ on the right hand side, where $s$ denotes the $(-\epsilon,0]$ coordinate. We define the {\bf symplectic completion\/} of $(X,\lambda)$ by
\[
\overline{(X,\lambda)} = X \sqcup_Y ([0,\infty)\times Y),
\]
glued using the above neighborhood identification. This has a symplectic form $\overline{\omega}$ which agrees with $\omega$ on $X$ and with $d(e^s\lambda)$ on $[0,\infty)\times Y$.

\begin{definition}
If $(X,\lambda)$ is a four-dimensional Liouville domain as above, an almost complex structure $J$ on $\overline{X}$ is {\bf cobordism-compatible\/} if:
\begin{itemize}
\item
On $X$, the almost complex structure $J$ is $\omega$-compatible.
\item
On $[0,\infty)\times Y$, the almost complex structure $J$ agrees with the restriction of an element of $\mathcal{J}(Y,\lambda|_Y)$.
\end{itemize}
\end{definition}

We say that a compact symplectic four-manifold $(X,\omega)$ is {\bf admissible\/} if each component is either closed, or a nondegenerate Liouville domain (with a choice of Liouville form which we omit from the notation). Let $\overline{X}$ denote the union of the closed components of $X$ and the symplectic completions of the Liouville components. Let $\mathcal{J}(\overline{X})$ denote the set of almost complex structures on $\overline{X}$ which are $\omega$-compatible on each closed component, and cobordism-compatible on each symplectic completion component.

\begin{definition}
\label{def:Liouvillecurve}
Suppose that $(X,\omega)$ is admissible and let $J\in\mathcal{J}(\overline{X})$. Define $\mathcal{M}^J\left(\overline{X}\right)$ to be the set of $J$-holomorphic curves
\[
u : (\Sigma,j) \longrightarrow (\overline{X},J)
\]
such that:
\begin{itemize}
\item The domain $\Sigma$ is a punctured compact surface (possibly disconnected).
\item The map $u$ is nonconstant on each component of $\Sigma$.
\item For each puncture of $\Sigma$, there is a Reeb orbit $\gamma$ on the boundary of some Liouville component of $X$, such that $u$ maps a neighborhood of the puncture asymptotically to $[0,\infty)\times\gamma$ as $s\to\infty$.
\end{itemize}
We declare two maps in $\mathcal{M}^J\left(\overline{X}\right)$ to be equivalent if they differ by a biholomorphism of the domains.
\end{definition}

If $u\in\mathcal{M}^J\left(\overline{X}\right)$, define its {\bf energy\/} $\mathcal{E}(u)$ to be the sum of the energy of each of its components, defined as follows. If $u$ maps to a closed symplectic manifold $(X,\omega)$, then $\mathcal{E}(u) = \int_\Sigma u^*\omega$. If $u$ maps to a symplectic completion, then $\mathcal{E}(u)$ is the sum, over all punctures of $\Sigma$, of the symplectic action of the corresponding Reeb orbit.

If $x_1,\ldots,x_k\in X$, define
\[
\mathcal{M}^J\left(\overline{X};x_1,\ldots,x_k\right) = \left\{u\in\mathcal{M}^J\left(\overline{X}\right) \;\big| \; x_1,\ldots,x_k \in u(\Sigma)\right\}.
\]

We are now ready to define the alternative ECH capacities.

\begin{definition}[\cite{altech}]
\label{def:ckalt}
Let $(X,\omega)$ be an admissible symplectic four-manifold. If $k$ is a nonnegative integer, define the {\bf alternative ECH capacity\/}
\begin{equation}
\label{eqn:defaltech}
c_k^{\op{Alt}}(X,\omega) = \sup_{\substack{J\in\mathcal{J}(\overline{X}) \\ x_1,\ldots,x_k\in X}} \inf_{u\in\mathcal{M}^J(\overline{X};x_1,\ldots,x_k)} \mathcal{E}(u) \in [0,+\infty].
\end{equation}
\end{definition}

\begin{remark}
\label{rem:maxmin}
If $(X,\omega)$ is a nondegenerate Liouville domain and $c_k^{\op{Alt}}(X,\omega)<\infty$, then one can replace `sup' and `inf' in \eqref{eqn:defaltech} by `max' and `min', because the action spectrum of a nondegenerate contact form is discrete. In this case one can think of $c_k^{\op{Alt}}(X,\omega)$ as a certificate guaranteeing the existence of low energy $J$-holomorphic curves through any $k$ points for any $J$.
\end{remark}

\begin{remark}
The max-min definition of $c_k^{\op{Alt}}$ was inspired by \cite{ms}, which defined capacities in arbitrary dimension by a similar construction using genus zero holomorphic curves which are constrained to have contact of order $k$ with a local divisor, with applications to stabilized symplectic embedding problems. Many variants of this max-min construction are possible, depending on which holomorphic curves and which constraints one considers.
\end{remark}

The following monotonicity property is crucial:

\begin{lemma}
\label{lem:ckaltmonotonicity}
(\cite[Lem.\ 3]{altech})
Let $(X,\omega)$ and $(X',\omega')$ be admissible symplectic four-manifolds. Suppose that there exists a symplectic embedding $(X,\omega) \hookrightarrow (X',\omega')$. Then
\begin{equation}
\label{eqn:ckaltmonotonicity}
c_k^{\op{Alt}}(X,\omega) \le c_k^{\op{Alt}}(X',\omega')
\end{equation}
for all $k$.
\end{lemma}

\begin{proof}
(outline)
The idea is to find a $J$-holomorphic curve through points $x_1,\ldots,x_k$ in $\overline{X}$, by starting with $J'$-holomorphic curves in $\overline{X'}$ through (the images of) $x_1,\ldots,x_k$ with energy $\le c_k^{\op{Alt}}(X',\omega')$ for a sequence of suitable almost complex structures $J'$, and stretching the neck and using Gromov compactness for holomorphic curves as in \cite{dw,taubescompactness}. In the limit a $J$-holomorphic curve in $\overline{X}$ through $x_1,\ldots,x_k$ will break off, and it will have energy $\le c_k^{\op{Alt}}(X',\omega')$. The relative adjunction formula and asymptotic writhe bound in \cite{ir} are used to control the genus of the holomorphic curves in the compactness argument. See \cite{altech} for details.
\end{proof}

We can now extend the definition of $c_k^{\op{Alt}}$ to all symplectic four-manifolds by a standard trick:

\begin{definition}
\label{def:ckaltextended}
If $(X',\omega')$ is any symplectic four-manifold, define
\[
c_k^{\op{Alt}}(X',\omega') = \sup\left\{c_k^{\op{Alt}}(X,\omega)\right\},
\]
where the supremum is over admissible symplectic four-manifolds $(X,\omega)$ for which there exists a symplectic embedding $(X,\omega)\hookrightarrow (X',\omega')$.
\end{definition}

It follows from Lemma~\ref{lem:ckaltmonotonicity} that Definition~\ref{def:ckaltextended} agrees with the previous Definition~\ref{def:ckalt} when $(X',\omega')$ is admissible, and in this case one can compute $c_k^{\op{Alt}}(X',\omega')$ using any Liouville forms on the Liouville domain components of $X'$.

%%%%%%%%%%%%%%%%%%%%%%%%%%%%%%%%%%%%%%%

\subsection{Basic properties of alternative ECH capacities}

%%%%%%%%%%%%%%%%%%%%%%%%%%%%%%%%%%%%%%%

The following are some basic properties of the capacities $c_k^{\op{Alt}}$.

\begin{proposition}[\cite{altech}]
\label{prop:ckaltproperties}
The capacities $c_k^{\op{Alt}}$ of four-dimensional symplectic manifolds have the following properties:
\begin{description}
\item{(Monotonicity)}
If there exists a symplectic embedding $(X,\omega)\hookrightarrow (X',\omega')$, then
\[
c_k^{\op{Alt}}(X,\omega) \le c_k^{\op{Alt}}(X',\omega')
\]
for all $k$.
\item{(Increasing)}
The (in)equalities \eqref{eqn:ckaltincreasing} hold.
\item{(Disjoint Union)}
\[
c_k^{\op{Alt}}\left(\coprod_{i=1}^m(X_i,\omega_i)\right) = \max_{k_1+\cdots+k_m=k}\sum_{i=1}^m c_{k_i}^{\op{Alt}}(X_i,\omega_i).
\]
\item{(Conformality)}
If $r>0$, then $c_k^{\op{Alt}}(X,r\omega) = rc_k^{\op{Alt}}(X,\omega)$.
\item{($C^0$-Continuity)}
For each $k$, the capacity $c_k^{\op{Alt}}$ defines a continuous function on the set of star-shaped domains in $\R^4$ with respect to the Hausdorff metric on their boundaries\footnote{In \cite{altech} we neglected to say ``on their boundaries''. We thank Janko Latschev for pointing out this error.}. 
\item{(Spectrality)}
If $(X,\omega)$ is a four-dimensional Liouville domain with boundary $Y$, then for each $k$ with $c_k^{\op{Alt}}(X,\omega)<\infty$, there exists an orbit set $\alpha$ in $Y$, which is nullhomologous in $X$, such that $c_k^{\op{Alt}}(X,\omega) = \mathcal{A}(\alpha)$.
\item{(Ball)}
We have
\[
c_k^{\op{Alt}}(B^4(a)) = da,
\]
where $d$ is the unique nonnegative integer satisfying \eqref{eqn:uniqueinteger}.
\end{description}
\end{proposition}

\begin{proof}
{\em Monotonicity:\/} This follows from Lemma~\ref{lem:ckaltmonotonicity} and Definition~\ref{def:ckaltextended}.

{\em Increasing:\/} Immediate from the definition.

{\em Disjoint Union:\/} Immediate from the definition.

{\em Conformality:\/} By Definition~\ref{def:ckaltextended}, we can assume that $(X,\omega)$ is admissible. By the Disjoint Union property, we can also assume that $(X,\omega)$ is connected. 

If $(X,\omega)$ is closed, then Conformality follows from the definition of $c_k^{\op{Alt}}$, since $\mathcal{J}(X,\omega) = \mathcal{J}(X,r\omega)$.

If $(X,\omega)$ is a Liouville domain with Liouville form $\lambda$, then there is a canonical diffeomorphism $\overline{(X,\lambda)} = \overline{(X,r\lambda)}$. However we need to be careful with the almost complex structures, because $\mathcal{J}(\overline{(X,\lambda)}) \neq \mathcal{J}(\overline{(X,r\lambda)})$, since the Reeb vector fields of $\lambda$ and $r\lambda$ differ by scaling by $r$. Suppose we are given $J \in\mathcal{J}(\overline{(X,r\lambda)})$ and $x_1,\ldots,x_k\in X$. We can push $J$ forward by a diffeomorphism of $\overline{X}$, which fixes $X$ and shifts the $s$ coordinate on $[0,\infty)\times Y$, to obtain an $\overline{\omega}$-compatible almost complex structure on $\overline{X}$ which agrees with a $\lambda$-compatible almost complex structure on $[\epsilon,\infty)\times Y$ for some $\epsilon>0$. We can then push this forward by the time $-\epsilon$ flow of the Liouville vector field $V$, to obtain an almost complex structure in $\mathcal{J}(\overline{(X,\lambda)})$. Let $\phi$ denote the composition of these two diffeomorphisms. By the definition of $c_k^{\op{Alt}}$, there exists
\[
u\in\mathcal{M}^{\phi_*J}(\overline{(X,\lambda)};\phi(x_1),\ldots,\phi(x_k))
\]
with energy $\mathcal{E}(u) \le c_k^{\op{Alt}}(X,\lambda)$. Then
\[
\phi^{-1}\circ u\in\mathcal{M}^{J}(\overline{(X,r\lambda)};x_1,\ldots,x_k)
\]
has energy $\mathcal{E}(\phi^{-1}\circ u) \le r c_k^{\op{Alt}}(X,\lambda)$. Thus $c_k^{\op{Alt}}(X,r\lambda) \le r c_k^{\op{Alt}}(X,\lambda)$. Switching the roles of $\lambda$ and $r\lambda$ in the above argument proves the reverse inequality.

{\em $C^0$-Continuity:\/} It follows from the Monotonicity property that if $X\subset\R^4$ is a star-shaped domain and $r>0$, then $c_k^{\op{Alt}}(rX) = r^2 c_k^{\op{Alt}}(X)$. Together with Monotonicity, this implies $C^0$-Continuity.

{\em Spectrality:\/} If the Liouville domain $(X,\omega)$ is nondegenerate, then by equation \eqref{eqn:defaltech} and Remark~\ref{rem:maxmin}, for a Liouville form $\lambda$, there exists $J\in\mathcal{J}(\overline{(X,\lambda)})$ and $x_1,\ldots,x_k\in X$ and $u\in\mathcal{M}^J(\overline{X};x_1,\ldots,x_k)$ with $\mathcal{E}(u) = c_k^{\op{Alt}}(X,\omega)$. Then the orbit set $\alpha$ obtained by summing the covering multiplicities of the Reeb orbits corresponding to the punctures of $u$, cf.\ \S\ref{sec:fredholmindex}, satisfies $c_k^{\op{Alt}}(X,\omega) = \mathcal{A}(\alpha)$. The orbit set $\alpha$ is nullhomologous in $X$ because of the existence of the $J$-holomorphic curve $u$. This proves Spectrality in the nondegenerate case.

To prove Spectrality in the degenerate case, one approximates $(X,\omega)$ by nondegenerate Liouville domains and uses a limiting argument; see \cite{altech} for details.

{\em Ball:\/} The proof is deferred to \S\ref{sec:ball}.
\end{proof}

\begin{remark}
Unlike with the elementary spectral invariants of contact three-manifolds, it is not known whether the alternative ECH capacities of symplectic four-manifolds are always finite. It is shown in \cite{beiner} that for some four-dimensional Liouville domains, the original ECH capacities $c_k^{\op{ECH}}$ are infinite for $k>0$. Note that one always has $c_k^{\op{Alt}} \le c_k^{\op{ECH}}$, by \cite[Thm.\ 12]{altech}, and examples are known where this inequality is strict \cite[Rem.\ 14]{altech}.
\end{remark}

\begin{remark}
\label{rem:ckalthd}
One can also define a version of $c_k^{\op{Alt}}$ for symplectic manifolds in higher dimensions. It will not immediately work to copy Definition~\ref{def:ckalt} directly, because the genus control in the proof of monotonicity in Lemma~\ref{lem:ckaltmonotonicity} uses the adjunction formula which is special to four dimensions. Instead, for each $g\ge 0$, one can define $c_{k,g}(X,\omega)$ as in Definitions~\ref{def:ckalt} and \ref{def:ckaltextended}, with the modification that for each holomorphic curve $u$, one requires the sum of the genera of all components of the domain to be at most $g$. Then $c_{k,g}$ is a symplectic capacity for $2n$-dimensional symplectic manifolds, and
\[
c_{k,0}(X,\omega) \ge c_{k,1}(X,\omega) \ge \cdots.
\]
One can define
\[
c_{k,\infty}(X,\omega) = \inf_{g\ge 0} c_{k,g}(X,\omega),
\]
and this is also a symplectic capacity. In the four-dimensional case, it follows from the definitions that
\[
c_{k,\infty}(X,\omega) \ge c_k^{\op{Alt}}(X,\omega).
\]
The relative adjunction formula implies that the reverse inequality holds for star-shaped domains in $\R^4$. In the four-dimensional case, the capacities $c_{k,g}$ sometimes give stronger symplectic embedding obstructions than $c_k^{\op{Alt}}$, similarly to \cite{beyond}.
\end{remark}

%%%%%%%%%%%%%%%%%%%%%%%%%%%%%%%%%%%%%%%%%%%%%%%%%%%%%%%%%%

\subsection{The Weyl law for domains in $\R^4$}
\label{sec:weyldomain}

%%%%%%%%%%%%%%%%%%%%%%%%%%%%%%%%%%%%%%%%%%%%%%%%%%%%%%%%%%%

We now show that the alternative ECH capacities of at least some symplectic four-manifolds $(X,\omega)$ satisfy the Weyl law
\begin{equation}
\label{eqn:ckaltweyl}
\lim_{k\to\infty}\frac{c_k^{\op{Alt}}(X,\omega)^2}{k} = 4\op{vol}(X,\omega).
\end{equation}
Here our convention for the volume of a four-dimensional symplectic manifold is
\begin{equation}
\label{eqn:symplecticvolume}
\op{vol}(X,\omega) = \frac{1}{2}\int_X\omega\wedge\omega,
\end{equation}
which agrees with the Euclidean volume for domains in $\R^4$.

\begin{theorem}
\label{thm:ckaltweyl}
Let $X\subset \R^4$ be a compact domain with piecewise smooth boundary. Then $X$ satisfies the Weyl law \eqref{eqn:ckaltweyl}.
\end{theorem}

The proof below, following the argument in \cite{qech}, uses only the computation of the capacities of the ball, together with the Disjoint Union and Monotonicity properties.

\begin{lemma}
\label{lem:ballweyl}
The ball $B^4(a)$ satisfies the Weyl law \eqref{eqn:ckaltweyl}.
\end{lemma}

\begin{proof}
This follows from the Ball property in Theorem~\ref{prop:ckaltproperties} and the fact that $\op{vol}(B^4(a))=a^2/2$.
\end{proof}

\begin{lemma}
\label{lem:unionweyl}
If $(X_i,\omega_i)$ are symplectic four-manifolds satisfying the Weyl law \eqref{eqn:ckaltweyl} for $i=1,\ldots,m$, then $(X,\omega) =\coprod_{i=1}^m(X_i,\omega_i)$ also satisfies the Weyl law \eqref{eqn:ckaltweyl}.
\end{lemma}

\begin{proof}
Write $V_i=\op{vol}(X_i,\omega_i)$ and $V=\op{vol}(X,\omega) = \sum_{i=1}^mV_i$.
Since each $(X_i,\omega_i)$ satisfies the Weyl law, if $k_i$ is a nonnegative integer then
\[
c_{k_i}(X_i,\omega_i) = 2\sqrt{k_iV_i} + o(k_i^{1/2}).
\]
It then follows from the Disjoint Union property in Proposition~\ref{prop:ckaltproperties} that
\[
c_k(X,\omega) = \max_{k_1+\cdots+k_m=k}\sum_{i=1}^m 2\sqrt{k_iV_i} + o(k^{1/2}).
\]
In the above maximum, each $k_i$ is an integer; but if we drop this requirement, which will introduce an $O(1)$ error, then the maximum is achieved when
\[
k_i = \frac{V_i}{V}k.
\]
We then get
\[
\begin{split}
c_k(X,\omega) &= \sum_{i=1}^m 2\sqrt{\frac{V_i^2}{V}k} + o(k^{1/2})\\
&= 2\sqrt{kV} + o(k^{1/2}),
\end{split}
\]
which is the Weyl law for $(X,\omega)$.
\end{proof}

\begin{proposition}
\label{prop:ckaltlowerweyl}
If $(X,\omega)$ is any compact symplectic four-manifold, possibly with boundary, then
\[
\liminf_{k\to\infty}\frac{c_k^{\op{Alt}}(X,\omega)^2}{k} \ge 4\op{vol}(X,\omega).
\]
\end{proposition}

\begin{proof}
For any $\epsilon>0$, we can find a symplectic embedding of a disjoint union $(X',\omega')$ of finitely many balls into $X$, whose image has volume $\ge \op{vol}(X,\omega)-\epsilon$. By Lemmas~\ref{lem:ballweyl} and \ref{lem:unionweyl}, $(X',\omega')$ satisfies the Weyl law \eqref{eqn:ckaltweyl}. By the Monotonicity property in Proposition~\ref{prop:ckaltproperties}, we get
\[
\begin{split}
c_k^{\op{Alt}}(X,\omega) &\ge c_k^{\op{Alt}}(X',\omega')\\
&\ge 2\sqrt{k(\op{vol}(X,\omega)-\epsilon)} + o(k^{1/2}).
\end{split}
\]
Since $\epsilon>0$ was arbitrary, this proves the proposition.
\end{proof}

\begin{proof}[Proof of Theorem~\ref{thm:ckaltweyl}]
By Proposition~\ref{prop:ckaltlowerweyl}, we just need to show that
\begin{equation}
\label{eqn:upperweyl}
\limsup_{k\to\infty}\frac{c_k^{\op{Alt}}(X,\omega)^2}{k} \le 4\op{vol}(X,\omega).
\end{equation}
To prove this, choose a large ball $B$ containing $X$. Write $V=\op{vol}(X)$ and $V'=\op{vol}(B)$. Given $k$, let $l$ be an integer such that
\[
k+l = \frac{V'}{V}k + O(1).
\]
By the Monotonicity and Disjoint Union properties\footnote{Strictly speaking we are applying these properties to slight shrinkings of $X$ and $B\setminus X$, whose closures are disjoint, and then using $C^0$-Continuity.} in Proposition~\ref{prop:ckaltproperties}, and using Proposition~\ref{prop:ckaltlowerweyl} applied to $B \setminus X$, we have
\begin{equation}
\label{eqn:klcalculation}
\begin{split}
c_k^{\op{Alt}}(X) &\le c_{k+l}^{\op{Alt}}(B) - c_l^{\op{Alt}}(B\setminus X)\\
& \le 2\sqrt{(k+l)V'} - 2\sqrt{l(V'-V)} + o(k^{1/2})\\
&= 2\sqrt{kV} + o(k^{1/2}).
\end{split}
\end{equation}
The inequality \eqref{eqn:upperweyl} follows.
\end{proof}

\begin{remark}
A more careful version of this argument in \cite{ruelle} shows\footnote{The result in \cite{ruelle} is stated for the original ECH capacities $c_k^{\op{ECH}}$. However the proof in \cite{ruelle} uses only the Monotonicity, Disjoint Union, and Ball properties, which hold for both $c_k^{\op{ECH}}$ and $c_k^{\op{Alt}}$.} that if $X\subset\R^4$ is a compact domain with piecewise smooth boundary, then
\[
c_k^{\op{Alt}}(X) = 2\sqrt{k\op{vol}(X)} + O(k^{1/4}).
\]
The argument is extended in \cite{cghind} to more general domains in $\R^4$, with the $O(k^{1/4})$ term replaced by a larger asymptotic error term depending on the inner Minkowski dimension of the boundary. On the other hand, Edtmair \cite{edtmairsubleading} used a more sophisticated argument to show that if $X\subset\R^4$ is a compact domain with smooth boundary, then
\[
c_k^{\op{Alt}}(X) = 2\sqrt{k\op{vol}(X)} + O(1).
\]
\end{remark}

\section{Computations of alternative ECH capacities}
\label{sec:comp}

We now compute some basic examples of alternative ECH capacities. The main examples we need for our applications are balls and ellipsoids, but we will also discuss the larger class of examples consisting of ``convex toric domains''.

To get upper bounds on alternative ECH capacities, we need existence results for holomorphic curves. Ultimately we will produce the holomorphic curves we need by starting with holomorphic curves in $\C P^2$ that are known to exist and then using neck stretching arguments.

To get lower bounds on alternative ECH capacities, we need to put constraints on the holomorphic curves that are guaranteed to exist. To do so we need to consider some topological aspects of holomorphic curves in four dimensions. Specifically, we need to review the ECH index inequality, in the context of Liouville domains. More details about the ECH index inequality, in more general situations, can be found in \cite{ir,bn}.

Below let $(X,\lambda)$ be a nondegenerate four-dimensional Liouville domain with boundary $Y$. Let $\xi=\op{Ker}(\lambda|_Y)$ denote the associated contact structure on $Y$. 

\subsection{The Conley-Zehnder index}

Let $\gamma:\R/T\Z\to Y$ be a Reeb orbit, and let $\tau$ be a symplectic trivialization of the rank two symplectic vector bundle $\gamma^*\xi$. Recall that the {\bf Conley-Zehnder index\/} of $\gamma$ with respect to $\tau$, denoted by
\[
\op{CZ}_\tau(\gamma) \in \Z,
\]
is defined as follows. For $t\in\R$, define a $2\times 2$ symplectic matrix $A_t\in\op{Sp}(2,\R)$ to be the composition
\[
\R^2 \stackrel{\tau^{-1}}{\longrightarrow} \xi_{\gamma(0)} \stackrel{d\phi_t}{\longrightarrow} \xi_{\gamma(t)} \stackrel{\tau}{\longrightarrow} \R^2.
\]
In particular, $A_0=I$, and so the path $\{A_t\}_{t\in[0,T]}$ defines an element of the universal cover $\widetilde{\op{Sp}}(2,\R)$. We then define $\op{CZ}_\tau(\gamma)$ to be the Conley-Zehnder index of this path of symplectic matrices as defined in \cite{cz,salamon}.

In our low dimensional case this can be described in terms of dynamical rotation numbers as follows. There is a homomorphism from $\op{Sp}(2,\R)$ to the group $\op{Diff}^+(S^1)$ of orientation-preserving diffeomorphisms of $S^1$, sending $A\in\op{Sp}(2,\R)$ to the map sending $e^{2\pi i\theta} \mapsto Ae^{2\pi i\theta}/\|Ae^{2\pi i\theta}\|$. This homomorphism lifts to a homomorphism
\[
\Phi:\widetilde{\op{Sp}}(2,\R) \longrightarrow \widetilde{\op{Diff}}^+(S^1).
\]
Next, the {\bf dynamical rotation number\/} is a quasimorphism
\[
\op{rot} : \widetilde{\op{Diff}}^+(S^1)\longrightarrow \R
\]
defined as follows: An element of $\widetilde{\op{Diff}}^+(S^1)$ is equivalent to a diffeomorphism $f:\R\to\R$ such that $f(x+1)=f(x)+1$ for all $x\in\R$. We then define
\[
\op{rot}(f) = \lim_{n\to\infty}\frac{f^n(x)-x}{n}\in\R,
\]
which is independent of the choice of $x\in\R$. We now define the {\bf rotation number\/} of $\gamma$ with respect to $\tau$ by
\[
\op{rot}_\tau(\gamma) = \op{rot}(\Phi(\{A_t\}_{t\in[0,T]}))\in\R.
\]
Finally, we have
\begin{equation}
\label{eqn:defCZ}
\op{CZ}_\tau(\gamma) = \floor{\op{rot}_\tau(\gamma)} + \ceil{\op{rot}_\tau(\gamma)}.
\end{equation}

Explicitly, there are three possibilities:
\begin{itemize}
\item
If the eigenvalues of $P_\gamma$ are on the unit circle, we say that $\gamma$ is {\bf elliptic\/}. In this case, there is an irrational\footnote{Our assumption that $\lambda$ is nondegenerate implies that $\gamma$ and all of its multiple covers are nondegenerate, which implies that $\theta$ must be irrational.} number $\theta\in\R$ such that $P_\gamma$ has eigenvalues $e^{\pm 2\pi i\theta}$, and
\[
\op{rot}_\tau(\gamma)=\theta, \quad\quad\quad \op{CZ}_\tau(\gamma) = \floor{\theta}+\ceil{\theta}.
\]
\item
If the eigenvalues of $P_\gamma$ are real and positive, we say that $\gamma$ is {\bf positive hyperbolic\/}. In this case there is an integer $k$ such that
\[
\op{rot}_\tau(\gamma)=k, \quad\quad\quad \op{CZ}_\tau(\gamma)=2k.
\]
The integer $k$ is the winding number of the path $\{A_t(v)\}_{t\in[0,T]}$ where $v$ is an eigenvector of $P_\gamma$.
\item
If the eigenvalues of $P_\gamma$ are real and negative, we say that $\gamma$ is {\bf negative hyperbolic\/}. In this case there is an integer $k$ such that
\[
\op{rot}_\tau(\gamma)=k+\frac{1}{2}, \quad\quad\quad \op{CZ}_\tau(\gamma)=2k+1.
\]
If $v$ is an eigenvector of $P_\gamma$, then the path $\{A_t(v)\}_{t\in[0,2T]}$ has winding number $2k+1$.
\end{itemize}

The rotation number and Conley-Zehnder index of a Reeb orbit $\gamma$ depend only on the homotopy class of symplectic trivialization $\tau$ of $\gamma^*\xi$.

\subsection{The relative first Chern class}

Let $\alpha=\{(\alpha_i,m_i)\}$ be an orbit set in the contact three-manifold $Y$, and write
\[
[\alpha]=\sum_im_i[\alpha_i]\in H_1(Y).
\]
Assume that $\alpha$ is nullhomologous in the four-dimensional Liouville domain $X$, i.e. $[\alpha]$ is in the kernel of the map $H_1(Y)\to H_1(X)$ induced by the inclusion. Let $H_2(X,\alpha,\emptyset)$ denote\footnote{This is a special notation from the ECH literature. In terms of relative homology, if $A$ denotes the union of (the image of) the orbits $\alpha_i$ in $Y$, then $H_2(X,\alpha,\emptyset)$ is the set of relative homology classes $Z\in H_2(X,A)$ such that $\partial Z = \sum_im_i[\alpha_i] \in H_1(A)$.} the set of $2$-chains $Z$ in $X$ with $\partial Z=\sum_im_i\alpha_i$, modulo boundaries of three-chains. This is an affine space over $H_2(X)$. 

If $Z\in H_2(X,\alpha,\emptyset)$, and if $\tau$ is a (homotopy class of) symplectic trivialization of $\xi$ over the Reeb orbits $\alpha_i$, then the {\bf relative first Chern class\/} of $TX$ over $Z$ relative to $\tau$ is an integer
\[
c_\tau(Z)\in\Z
\]
defined as follows. A cobordism-compatible almost complex structure $J$ on $\overline{X}$ gives $TX$ the structure of a rank two complex vector bundle over $X$, and $\tau$ induces a trivialization of this bundle over $\alpha$. Now we can represent the class $Z$ by a map $u:\Sigma\to X$ where $\Sigma$ is a compact oriented surface with boundary. Let $\psi$ be a generic section of the complex line bundle $u^*\det(TX)$ which on $\partial\Sigma$ is nonvanishing and constant with respect to $\tau$. We then define $c_\tau(Z)=\#\psi^{-1}(0)$, where $\#$ denotes the algebraic count.

\subsection{The Fredholm index}
\label{sec:fredholmindex}

Consider an almost complex structure $J\in\mathcal{J}(\overline{X})$. Each $J$-holomorphic curve
\[
u:\Sigma\longrightarrow \overline{X}
\]
in $\mathcal{M}^J(\overline{X})$ has an associated orbit set, obtained by taking the (possibly multiply covered) Reeb orbits associated to punctures of $\Sigma$ and summing their covering multiplicities. For example, if $\Sigma$ has two punctures, one of which is asymptotic to a simple Reeb orbit $\gamma$, and the other of which is asymptotic to a double cover of $\gamma$, then the associated orbit set is $\{(\gamma,3)\}$. We call this the {\bf asymptotic orbit set\/} of $u$.

If $\alpha$ is an orbit set, let $\mathcal{M}^J(\overline{X},\alpha)$ denote the set of $u\in\mathcal{M}^J(\overline{X})$ with asymptotic orbit set $\alpha$. Each $u\in\mathcal{M}^J(\overline{X},\alpha)$ determines a relative homology class
\[
[u]\in H_2(X,\alpha,\emptyset).
\]
If $\tau$ is a symplectic trivialization of $\xi$ over the Reeb orbits in $\alpha$, we let
\[
c_\tau(u) = c_\tau([u])\in\Z
\]
denote the relative first Chern class. If $\gamma_1,\ldots,\gamma_m$ are the (possibly multiply covered) Reeb orbits associated to the punctures of $u$, we define the {\bf Fredholm index\/}\footnote{In a $2n$-dimensional Liouville domain, the term $-\chi(\Sigma)$ is replaced by $(n-3)\chi(\Sigma)$.}
\[
\op{ind}(u) = -\chi(\Sigma) + 2c_\tau(u) + \sum_{i=1}^m\op{CZ}_\tau(\gamma_i).
\]
The Fredholm index does not depend on the choice of trivialization $\tau$.

Suppose that the domain $\Sigma$ is connected. We say that $u$ is {\bf multiply covered\/} if it factors through a branched cover $\Sigma\to\Sigma'$ of degree greater than one. Otherwise we say that $u$ is {\bf simple\/}. It is shown in \cite{wendl} that if $u$ is simple, then it is an embedding except possibly for finitely many singular points and/or self-intersections. The same conclusion (embedded except for finitely many points) holds when the domain of $u$ is disconnected, as long as the restriction of $u$ to each component is simple, and no two components have the same image. In this case we say that $u$ has {\bf no multiply covered components\/}.

If $J$ is generic, then for every simple $J$-holomorphic curve $u$, the moduli space $\mathcal{M}^J(\overline{X})$ is naturally a smooth manifold near $u$ of dimension $\op{ind}(u)$; see \cite{dragnev,wendl}.

\subsection{The ECH index}

Let $\alpha=\{(\alpha_i,m_i)\}$ be an orbit set in $Y$ which is nullhomologous in $X$, and let $Z\in H_2(X,\alpha,\emptyset)$. We define the {\bf ECH index\/}
\begin{equation}
\label{eqn:ECHindex}
I(Z) = c_\tau(Z) + Q_\tau(Z) + \sum_i\sum_{k=1}^{m_i}\op{CZ}_\tau(\alpha_i^k).
\end{equation}
Here $\tau$ is a homotopy class of symplectic trivialization of the contact structure $\xi$ over the Reeb orbits $\alpha_i$, and $\alpha_i^k$ denotes the Reeb orbit which is a $k$-fold cover of the simple Reeb orbit $\alpha_i$. The new term $Q_\tau(Z)$ is the {\bf relative intersection pairing\/} defined as follows.

Let $S\subset X$ be a smooth compact surface, embedded except at the boundary, and transverse to $Y$, representing the relative homology class $Z$. Let $S'$ be another such representative, whose interior intersects $S$ transversely at finitely many points. Then
\[
Q_\tau(Z) = \#(S\cap S') - \ell_\tau(S,S').
\]
Here $\#$ denotes the signed count of intersections, while $\ell_\tau$ is the asymptotic linking number defined in \cite[\S2.7]{ir}. It is shown in \cite{ir} that $Q_\tau(Z)$ does not depend on the choice of $S$ and $S'$, and $I(Z)$ does not depend on the choice of $\tau$.

The key fact about the ECH index is the following {\bf index inequality\/} \cite{pfh2,ir,bn}: If $u\in\mathcal{M}^J(\overline{X})$ has no multiply covered components, then\footnote{Equality holds in the index inequality only if $u$ is embedded (hence the `embedded' in embedded contact homology), and certain ``partition conditions'' hold. See e.g.\ \cite[\S3.9]{bn}.}
\begin{equation}
\label{eqn:indexinequality}
\op{ind}(u) \le I([u]).
\end{equation}

\subsection{Enhanced spectrality}

The reason we have introduced the ECH index is the following proposition. This result says more about the Spectrality property in Proposition~\ref{prop:ckaltproperties} for star-shaped domains, and leads to lower bounds on their alternative ECH capacities.

\begin{proposition}[\cite{altech}]
\label{prop:enhanced}
Suppose $X\subset\R^4$ is a nondegenerate star-shaped domain and $k$ is a nonnegative integer. Then there exists an orbit set $\alpha$ in $\partial X$ such that $c_k^{\op{Alt}}(X) = \mathcal{A}(\alpha)$ and the ECH index $I(\alpha) \ge 2k$.
\end{proposition}

\begin{proof}
Given an orbit set $\alpha$, an almost complex structure $J\in\mathcal{J}(\overline{X})$, and points $x_1,\ldots,x_k$, let 
\[
\mathcal{M}^J\left(\overline{X},\alpha;x_1,\ldots,x_k\right) = \left\{u\in\mathcal{M}^J\left(\overline{X},\alpha\right) \;\big| \; x_1,\ldots,x_k \in u(\Sigma)\right\}.
\]
By Remark~\ref{rem:maxmin}, we have
\[
c_k^{\op{Alt}}(X,\omega) = \max_{\substack{J\in\mathcal{J}(\overline{X}) \\ x_1,\ldots,x_k\in X}} \min\left\{\mathcal{A}(\alpha) \;\big|\; \mathcal{M}^J(\overline{X},\alpha;x_1,\ldots,x_k)\neq \emptyset\right\}.
\]
In the above equation, the minimum must be realized by a curve with no multiply covered components, because we can replace any multiply covered components by their underlying simple curves so as to decrease energy without violating the point constraints. Thus we have
\begin{equation}
\label{eqn:maxminstar}
c_k^{\op{Alt}}(X,\omega) = \max_{\substack{J\in\mathcal{J}(\overline{X}) \\ x_1,\ldots,x_k\in X}} \min\left\{\mathcal{A}(\alpha) \;\big|\; \mathcal{M}^J_*(\overline{X},\alpha;x_1,\ldots,x_k)\neq \emptyset\right\},
\end{equation}
where $\mathcal{M}^J_*(\cdots)$ denotes the set of curves in $\mathcal{M}^J(\cdots)$ with no multiply covered components.

Suppose that $J$ is generic, so that holomorphic curves with no multiply covered components are cut out transversely, i.e. live in moduli spaces which are smooth manifolds with dimension equal to the Fredholm index. Suppose also that the points $x_1,\ldots,x_k$ are generic. Then the minimum in \eqref{eqn:maxminstar} must be realized by an orbit set $\alpha$ with ECH index $I(\alpha) \ge 2k$. This is because if $I(\alpha)<2k$, then by the index inequality \eqref{eqn:indexinequality}, all components of $\mathcal{M}^J_*(\overline{X},\alpha)$ have dimension less than $2k$, and so cannot include curves passing through $k$ generic points.

By a Gromov compactness argument, the maximum in \eqref{eqn:maxminstar} must be realized by generic $J,x_1,\ldots,x_k$.
\end{proof}

To make use of Proposition~\ref{prop:enhanced}, we now discuss how to compute the ECH index.

\subsection{The ECH index in star-shaped domains}

Suppose that $X\subset\R^4$ is a nondegenerate star-shaped domain with boundary $Y$. In this case the formula \eqref{eqn:ECHindex} for the ECH index can be made more explicit as follows.

To start, every orbit set $\alpha=\{(\alpha_i,m_i)\}$ is nullhomologous in both $X$ and $Y$, and the set $H_2(X,\alpha,\emptyset)$ has just one element. Write $I(\alpha)$, $c_\tau(\alpha)$, and $Q_\tau(\alpha)$ to denote $I(Z)$, $c_\tau(Z)$, and $Q_\tau(Z)$ respectively, where $Z$ is the unique element of $H_2(X,\alpha,\emptyset)$.

The relative first Chern class is linear in the sense that
\[
c_\tau(\alpha) = \sum_im_ic_\tau(\alpha_i).
\]
The relative self-intersection pairing is quadratic in the sense that
\begin{equation}
\label{eqn:qtauquadratic}
Q_\tau(\alpha) = \sum_im_i^2Q_\tau(\alpha_i) + \sum_{i\neq j}m_im_j\ell(\alpha_i,\alpha_j),
\end{equation}
where $\ell$ denotes the linking number of oriented knots in $S^3$. For an individual simple Reeb orbit $\gamma$, the number $Q_\tau(\gamma)$ is the linking number of $\gamma$ with a pushoff of $\gamma$ via the framing determined by $\tau$. An equivalent statement is that
\[
Q_\tau(\gamma) - c_\tau(\gamma) = \op{sl}(\gamma)
\]
where $\op{sl}$ denotes the self-linking number. Here if $K$ is any knot in $S^3$ transverse to the contact structure $\xi$, then $\op{sl}(K)$ is defined as follows: Let $\Sigma$ be a Seifert surface for $K$, and let $V$ be a nonvanishing section of $\xi$ over $\Sigma$. Then $\op{sl}(K) = \ell(K,K')$, where $K'$ is a pushoff of $K$ in the direction $V$.

Combining the above three equations with equation \eqref{eqn:ECHindex}, we get
\begin{equation}
\label{eqn:Isimplifiedless}
I(\alpha) = \sum_i(m_i^2+m_i)c_\tau(\alpha_i) + \sum_im_i^2\op{sl}(\alpha_i) + \sum_{i\neq j}m_im_j\ell(\alpha_i,\alpha_j) + \sum_i\sum_{k=1}^{m_i}\op{CZ}_\tau(\alpha_i^k).
\end{equation}
We can simplify the above formula as follows. There is a unique homotopy class of global trivialization $\tau$ of $\xi$ over $Y$. With this trivialization, all relative first Chern classes are zero. Let $\op{CZ}$ denote the Conley-Zehnder index with respect to the global trivialization. Then we get
\begin{equation}
\label{eqn:Isimplifiedmore}
I(\alpha) = \sum_im_i^2\op{sl}(\alpha_i) + \sum_{i\neq j}m_im_j\ell(\alpha_i,\alpha_j) + \sum_i\sum_{k=1}^{m_i}\op{CZ}(\alpha_i^k).
\end{equation}

\subsection{The ECH index in an ellipsoid}

We now consider the case where $X$ is an ellipsoid $E(a,b)$. We assume that $a/b$ is irrational, which is equivalent to nondegeneracy of the contact form on the boundary.

Recall from Example~\ref{ex:ellipsoid} that there are two simple Reeb orbits $\gamma_1$ and $\gamma_2$, which have symplectic actions $a$ and $b$ respectively. Thus the symplectic actions of orbit sets consist of the numbers $N_k(a,b)$ defined in \S\ref{sec:propspec}.

\begin{lemma}
\label{lem:ellipsoidIA}
\cite[\S3.7]{bn}
Let $a,b$ be positive real numbers with $a/b$ irrational, and let $\alpha$ be an orbit set in $\partial E(a,b)$. Then
\[
I(\alpha) = 2k \Longleftrightarrow \mathcal{A}(\alpha) = N_k(a,b).
\]
\end{lemma}

\begin{proof}[Proof (outline)] We proceed in two steps.

{\em Step 1.\/} We first compute the ECH index.

The simple Reeb orbits $\gamma_1$ and $\gamma_2$ have (self-)linking numbers given by
\[
\op{sl}(\gamma_1) = -1, \quad\quad\quad \ell(\gamma_1,\gamma_2) = 1, \quad\quad\quad \op{sl}(\gamma_2) = -1.
\]

By equation \eqref{eqn:ellipsoidReeb}, the Reeb orbits $\gamma_1$ and $\gamma_2$ are elliptic. The bundle $\xi|_{\gamma_1}$ is naturally identified with the second $\C$ factor in $\C^2$, and this determines a trivialization $\tau$ with respect to which $\gamma_1$ has rotation number $a/b$. Likewise, $\xi_{\gamma_2}$ is naturally identified with the first $\C$ factor in $\C^2$, and this determines a trivialization $\tau$ with respect to which $\gamma_2$ has rotation number $b/a$. Thus with this trivialization we have
\[
\op{CZ}_\tau(\gamma_1^j) = 2\floor{\frac{ja}{b}}+1, \quad\quad\quad \op{CZ}_\tau(\gamma_2^j) = 2\floor{\frac{jb}{a}}+1.
\]
Note that this trivialization $\tau$ is not the restriction of the global trivialization. Rather we have
\[
c_\tau(\gamma_i) = 1, \quad\quad\quad Q_\tau(\gamma_i) = 0.
\]

We can write an orbit set using the multiplicative notation $\alpha = \gamma_1^{m_1}\gamma_2^{m_2}$. This means that when $m_1>0$ we include the pair $(\gamma_1,m_1)$, and when $m_2>0$ we include the pair $(\gamma_2,m_2)$. Putting the above calculations into equation \eqref{eqn:Isimplifiedless}, we get
\begin{equation}
\label{eqn:ellipsoidI}
I\left(\gamma_1^{m_1}\gamma_2^{m_2}\right) = 2\left(m_1+m_2+m_1m_2 + \sum_{j=1}^{m_1}\floor{\frac{ja}{b}} + \sum_{j=1}^{m_2}\floor{\frac{jb}{a}}\right).
\end{equation}

{\em Step 2.\/} We now read off the conclusion of the lemma from equation \eqref{eqn:ellipsoidI}.

Let $T$ be the triangle in the plane whose edges are segments of the $x$ axis, the $y$ axis, and the line through the point $(m_1,m_2)$ with slope $-a/b$. Let $k$ denote the number of lattice points enclosed by this triangle, including lattice points on the boundary, minus one. By dividing this triangle into a rectangle with corners at $(0,0), (m_1,0), (0,m_2), (m_1,m_2)$ and two smaller triangles, we see from equation \eqref{eqn:ellipsoidI} that
\[
I\left(\gamma_1^{m_1}\gamma_2^{m_2}\right) = 2k.
\]
On the other hand, as in Example~\ref{ex:ellipsoidweyl}, we have
\[
\mathcal{A}\left(\gamma_1^{m_1}\gamma_2^{m_2}\right) = N_k(a,b).
\]
\end{proof}

\begin{corollary}
\label{cor:ckaltellipsoidlb}
If $a,b>0$, then
\begin{equation}
\label{eqn:ckaltellipsoidlb}
c_k^{\op{Alt}}(E(a,b)) \ge N_k(a,b).
\end{equation}
\end{corollary}

\begin{proof}
If $a/b$ is irrational, then this follows from Proposition~\ref{prop:enhanced} and Lemma~\ref{lem:ellipsoidIA}. If $a/b$ is rational, then approximate $E(a,b)$ by irrational ellipsoids inside it, and use the irrational case and the Monotonicity property in Proposition~\ref{prop:ckaltproperties}.
\end{proof}

We will see in \S\ref{sec:ckaltconvex} below that the inequality \eqref{eqn:ckaltellipsoidlb} is in fact an equality.

%%%%%%%%%%%%%%%%%%%%%%%%%%%%%%%%%%%%%%%%%

\subsection{Alternative ECH capacities of a ball}
\label{sec:ball}

We can now prove:

\begin{proposition}
\label{prop:ball}
The alternative ECH capacities of a ball are given by
\[
c_k^{\op{Alt}}\left(B^4(a)\right) = da
\]
where $d$ is the unique nonnegative integer satisfying \eqref{eqn:uniqueinteger}.
\end{proposition}

\begin{proof}
The inequality $c_k^{\op{Alt}}(B^4(a)) \ge da$ follows from Corollary~\ref{cor:ckaltellipsoidlb}, because $da=N_k(a,a)$.

To prove the reverse inequality, let $\C P^2(a)$ denote $\C P^2$ equipped with the Fubini-Study symplectic form, scaled so that a line has symplectic area $a$. If $a<a'$ then there is a symplectic embedding $B^4(a) \hookrightarrow \C P^2(a')$. So by the Monotonicity property in Proposition~\ref{prop:ckaltproperties}, it is enough to show that
\[
c_k^{\op{Alt}}(\C P^2(a)) \le da.
\]
To do so, by the definition of $c_k^{\op{Alt}}$, it is enough to show that for any $J\in\mathcal{J}(\C P^2)$, and for any $k$ points $x_1,\ldots,x_k\in \C P^2$ with $2k \le d^2+3d$, there exists a $J$-holomorphic curve (possibly with disconnected domain) of degree $d$ passing through the points $x_1,\ldots,x_k$. For a given $J$, for generic $x_1,\ldots,x_k$, this was proved by Gromov \cite[\S0.2.B]{gromov}. The general case then follows from Gromov compactness.
\end{proof}

%%%%%%%%%%%%%%%%%%%%%%%%%%%%%%%%%%

\subsection{Alternative ECH capacities of convex toric domains}
\label{sec:ckaltconvex}

We now outline how to compute the alternative ECH capacities for a larger class of examples.

Let $\Omega$ be a domain in $\R^2_{\ge 0}$. We define the {\bf toric domain\/}
\[
X_\Omega = \left\{z\in\C^2\;\big|\; \left(\pi|z_1|^2, \pi|z_2|^2\right) \in \Omega\right\}.
\]
We say that $X_\Omega$ is a {\bf convex toric domain\/} if $\Omega$ is compact and the set
\begin{equation}
\label{eqn:widehatomega}
\widehat{\Omega} = \left\{\mu\in\R^2\;\big|\;\left(|\mu_1|,|\mu_2|\right)\in\Omega\right\}
\end{equation}
is convex\footnote{A ``convex toric domain'' is itself a convex domain in $\R^4$. However a toric domain that is convex in $\R^4$ might not be a ``convex toric domain''. See \cite[\S2]{ghr} for explanation.} and contains $0$ in its interior. For example, if $a,b>0$ and if $\Omega$ is the triangular region with vertices $(0,0)$, $(a,0)$, and $(0,b)$, then $X_\Omega$ is the ellipsoid $E(a,b)$, which is a convex toric domain.

We now state a formula for the alternative ECH capacities of a convex toric domain. Let $\|\cdot\|_\Omega^*$ denote the dual norm of the norm on $\R^2$ for which $\widehat{\Omega}$ is the unit ball. Explicitly, if $v\in \R^2$, then
\[
\|v\|_\Omega^* = \max\left\{\langle v,w\rangle \;\big|\; w\in\widehat{\Omega}\right\}.
\]
Define a {\bf convex integral path\/} to be a path $\Lambda$ in $\R^2_{\ge 0}$ from $(0,b)$ to $(a,0)$ for some $a,b \in \Z_{\ge 0}$, which consists of line segments between lattice points, such that each edge has slope in the interval $[-\infty,0]$, and the curve consisting of $\Lambda$ together with the line segments from the origin to $(a,0)$ and $(0,b)$ is the boundary of a convex region. If $\Lambda$ is a convex integral path, let $\op{Edges}(\Lambda)$ denote the set of its edge vectors. Define the {\bf $\Omega$-length\/} of $\Lambda$ by
\begin{equation}
\label{eqn:omegalength}
\ell_\Omega(\Lambda) = \sum_{v\in\op{Edges}(\Lambda)}\|v^\perp\|_\Omega^*,
\end{equation}
where $v^\perp$ denotes a 90 degree rotation of $v$. Finally, define $\mathcal{L}(\Lambda)$ to be the number of lattice points in the region enclosed by $\Lambda$ and the axes, including lattice points on the boundary.

\begin{theorem}[\cite{altech}]
\label{thm:ckaltconvex}
If $X_\Omega\subset\R^4$ is a convex toric domain, and $k$ is a nonnegative integer, then
\begin{equation}
\label{eqn:ckaltconvex}
c_k^{\op{Alt}}(X_\Omega) = \min\left\{\ell_\Omega(\Lambda) \;\big|\; \mathcal{L}(\Lambda) = k+1\right\},
\end{equation}
where the minimum is over convex integral paths $\Lambda$.
\end{theorem}

Here is an example:

\begin{corollary}
\label{cor:ckaltellipsoid}
If $a,b>0$ then $c_k^{\op{Alt}}(E(a,b)) = N_k(a,b)$.
\end{corollary}

\begin{proof}
By Corollary~\ref{cor:ckaltellipsoidlb}, we just need to show that
\begin{equation}
\label{eqn:ckaltellipsoidub}
c_k^{\op{Alt}}(E(a,b))\le N_k(a,b).
\end{equation}
For simplicity suppose that $a/b$ is irrational; the proof in the general case is similar.

Let $\Omega$ be the triangular region with vertices $(0,0)$, $(a,0)$, and $(0,b)$, so that $E(a,b)=X_\Omega$. Observe that if $\Lambda$ is a convex integral path, then $\ell_\Omega(\Lambda)$ is computed as follows. Let $L\ge 0$ be the minimum such that every point $(x,y)$ on $\Lambda$ satisfies $bx+ay\le L$. In particular the line $bx+ay=L$ is a tangent line to the region enclosed by $\Lambda$ and the axes. Let $(n,m)\in \Z^2_{\ge 0}$ be the vertex of $\Lambda$ for which $bx+ay=L$. (This vertex is unique since $a/b$ is irrational.) Then it follows from \eqref{eqn:omegalength} that
\begin{equation}
\label{eqn:lomegaellipsoid}
\ell_\Omega(\Lambda) = am+bn.
\end{equation}

Now to prove \eqref{eqn:ckaltellipsoidub}, given $k$, let $L=N_k(a,b)$. Let $\Lambda$ be the maximal convex integral path such that every point $(x,y)$ on $\Lambda$ satisfies $bx+ay\le L$. Then by equation \eqref{eqn:lomegaellipsoid}, we have $\ell_\Omega(\Lambda) = N_k(a,b)$. On the other hand, as in Example~\ref{ex:ellipsoidweyl} we have $\mathcal{L}(\Lambda) = k+1$. The inequality \eqref{eqn:ckaltellipsoidub} now follows from \eqref{eqn:ckaltconvex}.
\end{proof}

\begin{proof}[Proof of Theorem~\ref{thm:ckaltconvex}]
We proceed in three steps.

{\em Step 1.\/} We first compute the symplectic action and ECH index of orbit sets in a suitable nondegenerate perturbation of $X_\Omega$.

Fix a nonnegative integer $k$, and choose $\epsilon>0$ which is small with respect to $k$. As reviewed\footnote{The original source for this calculation is \cite{beyond}. The review in \cite{anchored} uses a ``modernized'' notational convention for $e_{a,b}$ and $h_{a,b}$ which we follow here.} in \cite[\S2.2]{anchored}, we can approximate $X_\Omega$ by a nondegenerate star-shaped domain $X\subset \R^4$ with the following properties:
\begin{itemize}
\item
The $C^0$-distance between $\partial X_\Omega$ and $\partial X$ is at most $\epsilon$.
\item
Up to symplectic action $\epsilon^{-1}$, the simple Reeb orbits on $\partial X$ consist of an elliptic orbit $e_{a,b}$ and a positive hyperbolic $h_{a,b}$ for every pair of relatively prime nonnegative integers $a,b$, not both zero, with $\|(a,b)\|_\Omega^* < \epsilon^{-1}$. For $h_{a,b}$ we also require $a,b>0$.
\item
If $\gamma_{a,b}$ denotes either $e_{a,b}$ or $h_{a,b}$, then its symplectic action satisfies
\begin{equation}
\label{eqn:Agammaab}
\left|\mathcal{A}(\gamma_{a,b}) - \|(a,b)\|_\Omega^* \right| \le \epsilon.
\end{equation}
\item
If $\gamma_{a',b'}$ is distinct from $\gamma_{a,b}$, then the linking number
\begin{equation}
\label{eqn:ellab}
\ell(\gamma_{a,b},\gamma_{a',b'}) = \max(ab',a'b).
\end{equation}
\item
There is a trivialization $\tau$ of $\xi$ over the Reeb orbits $e_{a,b}$ and $h_{a,b}$ such that
\begin{align}
\label{eqn:ctauab}
c_\tau(\gamma_{a,b}) &= a+b,\\
\label{eqn:qtauab}
Q_\tau(\gamma_{a,b}) &= ab,\\
\label{eqn:cztaueab}
\op{CZ}_\tau(e_{a,b}^m) &= 1, \quad\quad\quad \mbox{if $m\|(a,b)\|_\Omega^* < \epsilon^{-1}$},\\
\label{eqn:cztauhab}
\op{CZ}_\tau(h_{a,b}) &= 0.
\end{align}
\end{itemize}

To an orbit set $\alpha$ in $\partial X$ with action less than $\epsilon^{-1}$, we associate a convex integral path $\Lambda$ as follows. If $\alpha$ includes orbits $e_{a,b}$ and/or $h_{a,b}$ with total multiplicity $m$, then $\Lambda$ contains an edge with corresponding vector $(mb,-ma)$. We then append all the edges in order of their slopes to obtain the path $\Lambda$.

Let $P_\Lambda$ denote the polygonal region enclosed by $\Lambda$ and the axes. By dividing the region $P_\Lambda$ into rectangles and triangles and using equations \eqref{eqn:qtauquadratic}, \eqref{eqn:ellab}, and \eqref{eqn:qtauab}, we find that
\[
Q_\tau(\alpha) = 2\op{Area}(P_\Lambda).
\]
If $(0,b(\Lambda))$ and $(a(\Lambda),0)$ denote the endpoints of the path $\Lambda$, then by equation \eqref{eqn:ctauab} we have
\[
c_\tau(\alpha) = a(\Lambda) + b(\Lambda).
\]
Let $m(\Lambda)$ denote the total multiplicity of the edges of $\Lambda$. (The multiplicity of an edge is the greatest common divisor of the components of the corresponding vector.) Then by \eqref{eqn:cztaueab} and \eqref{eqn:cztauhab}, the sum of the Conley-Zehnder index terms in the ECH index \eqref{eqn:ECHindex} is in the interval $[0,m(\Lambda)]$. Putting this together, we conclude that
\begin{equation}
\label{eqn:Iconvex}
2\op{Area}(P_\Lambda) + a(\Lambda) + b(\Lambda) \le I(\alpha) \le 2\op{Area}(P_\Lambda) + a(\Lambda) + b(\Lambda) + m(\Lambda).
\end{equation}
Now the sum of the last three terms on the right hand side is the number of lattice points on the boundary of $P_\Lambda$. So by Pick's theorem, we have
\begin{equation}
\label{eqn:ILambda}
I(\alpha) \le 2\left(\mathcal{L}(\Lambda) - 1\right).
\end{equation}
On the other hand, by \eqref{eqn:Agammaab}, we have
\begin{equation}
\label{eqn:ALambda}
\mathcal{A}(\alpha) = \ell_\Omega(\Lambda) + O(\epsilon).
\end{equation}

{\em Step 2.\/}
We now show that the left hand side of \eqref{eqn:ckaltconvex} is greater than or equal to the right hand side.

It follows from \eqref{eqn:ILambda} and \eqref{eqn:ALambda} and Proposition~\ref{prop:enhanced} that
\[
c_k^{\op{Alt}}(X) \ge \min\left\{\ell_\Omega(\Lambda) \;\big|\; \mathcal{L}(\Lambda) \ge k+1\right\} + O(\epsilon).
\]
By the $C^0$-Continuity property in Proposition~\ref{prop:ckaltproperties}, it follows that
\begin{equation}
\label{eqn:ckaltconvexlb}
c_k^{\op{Alt}}(X_\Omega) \ge \min\left\{\ell_\Omega(\Lambda) \;\big|\; \mathcal{L}(\Lambda) \ge k+1\right\}.
\end{equation}
Finally, we have
\begin{equation}
\label{eqn:roundingcorners}
\min\left\{\ell_\Omega(\Lambda) \;\big|\; \mathcal{L}(\Lambda) \ge k+1\right\} \ge \min\left\{\ell_\Omega(\Lambda) \;\big|\; \mathcal{L}(\Lambda) = k+1\right\},
\end{equation}
because given a convex integral path $\Lambda$ with $\mathcal{L}(\Lambda) > k+1$, one can ``round corners'' to reduce the number of enclosed lattice points without increasing the $\Omega$-length; compare \cite[Lem.\ 2.14]{t3}.

{\em Step 3.\/} We now review why the left hand side of \eqref{eqn:ckaltconvex} is less than or equal to the right hand side.

In \cite{cgconcaveconvex}, a ``weight expansion'' of $X_\Omega$ is defined, which is a (possibly finite or infinite) sequence of positive real numbers $(a;a_1,a_2,\ldots)$. This sequence has the property that there exists a symplectic embedding
\[
\op{int}\left(X_\Omega\right) \sqcup \coprod_i \op{int}\left(B^4(a_i)\right) \hookrightarrow B^4(a).
\]
It follows from the Disjoint Union and Monotonicity properties\footnote{Note here that passing from a domain to its interior does not change its capacities, by the Monotonicity and Conformality properties.} in Proposition~\ref{prop:ckaltproperties} that
\[
c_k^{\op{Alt}}\left(X_\Omega\right) \le \inf_{k+k_1+k_2+\cdots = l}\left(c_l^{\op{Alt}}\left(B^4(a)\right) - \sum_{i} c_{k_i}^{\op{Alt}}\left(B^4(a_i)\right)\right).
\]
It is shown\footnote{More precisely, the paper \cite{cgconcaveconvex} discusses the original ECH capacities $c_k^{\op{ECH}}$, but $c_k^{\op{ECH}}$ and $c_k^{\op{Alt}}$ are the same for a ball. Also, \cite{cgconcaveconvex} considers a more general notion of ``convex toric domain'', in which one only assumes that $\Omega$ is convex but not necessarily $\widehat{\Omega}$, together with a corresponding more general notion of ``convex integral path''. For our notion of convex toric domain, the minimum in \eqref{eqn:ckaltconvex} is the same for both notions of convex integral path, as explained in \cite[Prop.\ 5.6]{beyond}.} by a combinatorial argument in \cite[Appendix]{cgconcaveconvex} that if $(a;a_1,a_2,\ldots)$ is the weight expansion of $X$, then the right hand side of the above inequality agrees with the right hand side of \eqref{eqn:ckaltconvex}. 
\end{proof}

There is also a similar formula for the alternative ECH capacities of four-dimensional ``concave toric domains''; see \cite[Thm.\ 15]{altech}. For both convex toric domains and concave toric domains in four dimensions, the alternative ECH capacities $c_k^{\op{Alt}}$ agree with the original ECH capacities as computed in \cite{concave,cgconcaveconvex}. It was shown by McDuff \cite{mcduffellipsoid} that the ECH capacities give a sharp obstruction to symplectically embedding one four-dimensional open ellipsoid into another. More generally, Cristofaro-Gardiner \cite{cgconcaveconvex} showed that ECH capacities give a sharp obstruction to symplectically embedding an open concave toric domain into a convex toric domain in four dimensions. However ECH capacities do not always give sharp obstructions to symplectically embedding one open four-dimensional convex toric domain into another, and a more careful examination of holomorphic curves in \cite{beyond} leads to sharper obstructions in some cases. The holomorphic curves needed for the results in \cite{beyond} can also be obtained from the alternative ECH capacities.

\section{Elementary spectral invariants: definition and proofs}
\label{sec:defspec}

We now return to dynamics and explain how to define the elementary spectral invariants $c_k$ that were introduced in \S\ref{sec:spec}.

\subsection{Definition of elementary spectral invariants\/}
\label{subsec:defspec}

The following discussion parallels \S\ref{sec:defaltech}, with appropriate modifications for the contact case. We begin with a modification of Definition~\ref{def:lambdacompatible}:

\begin{definition}
\label{def:generalizedlambdacompatible}
Let $Y$ be a three-manifold, and let $\lambda$ be a contact form on $Y$. Define $\widetilde{\mathcal{J}}(Y,\lambda)$ to be the set of almost complex structures $J$ on $\R\times Y$ with the following properties:
\begin{itemize}
\item
There exists $J_+\in\mathcal{J}(Y,\lambda)$ such that $J$ agrees with $J_+$ on $[0,\infty)\times Y$.
\item
There exists $R_0\ge 0$ and $J_-\in\mathcal{J}(Y,\lambda)$ such that $J$ agrees with $J_-$ on $(-\infty,-R_0]\times Y$.
\item
$J$ is compatible with the symplectic form $d(e^s\lambda)$ on $[-R_0,0]\times Y$.
\end{itemize}
\end{definition}

We now have the following modification of Definition~\ref{def:Liouvillecurve}:

\begin{definition}
\label{def:Liouvillecurvemodified}
Let $(Y,\lambda)$ be a nondegenerate contact three-manifold, and let $J\in\widetilde{\mathcal{J}}(Y,\lambda)$. Define $\mathcal{M}^J(\R\times Y)$ to be the set of $J$-holomorphic curves
\[
u : (\Sigma,j) \longrightarrow (\R\times Y,J)
\]
such that:
\begin{itemize}
\item The domain $\Sigma$ is a punctured compact surface (possibly disconnected).
\item The map $u$ is nonconstant on each component of $\Sigma$.
\item For each puncture of $\Sigma$, there is a Reeb orbit $\gamma$ in $Y$, such that either $u$ maps a neighborhood of the puncture asymptotically to $[0,\infty)\times\gamma$ as $s\to\infty$ (this is a {\bf positive puncture\/}), or $u$ maps a neighborhood of the puncture asymptotically to $(-\infty,0]\times\gamma$ as $s\to-\infty$ (this is a {\bf negative puncture\/}).
\end{itemize}
We declare two maps in $\mathcal{M}^J(\R\times Y)$ to be equivalent if they differ by a biholomorphism of the domains.
\end{definition}

If $u\in\mathcal{M}^J(\R\times Y)$, define its {\bf positive energy\/} $\mathcal{E}_+(u)$ to be the sum, over all positive punctures of $\Sigma$, of the symplectic action of the corresponding Reeb orbit.

If $x_1,\ldots,x_k\in (-\infty,0] \times Y$, define
\[
\mathcal{M}^J(\R\times Y;x_1,\ldots,x_k) = \left\{u\in\mathcal{M}^J(\R\times Y) \;\big| \; x_1,\ldots,x_k \in u(\Sigma)\right\}.
\]

We now have the following modification of Definition~\ref{def:ckalt}:

\begin{definition}[\cite{altspec}]
\label{def:ck}
Let $Y$ be a closed three-manifold and let $\lambda$ be a nondegenerate contact form on $Y$. If $k$ is a nonnegative integer, define the {\bf elementary spectral invariant\/}
\begin{equation}
\label{eqn:defaltspec}
c_k(Y,\lambda) = \sup_{\substack{J\in\widetilde{\mathcal{J}}(Y,\lambda) \\ x_1,\ldots,x_k\in (-\infty,0]\times Y}} \inf_{u\in\mathcal{M}^J(\R\times Y;x_1,\ldots,x_k)} \mathcal{E}_+(u) \in [0,+\infty].
\end{equation}
\end{definition}

As in Remark~\ref{rem:maxmin}, if $c_k(Y,\lambda)<\infty$, then one can replace `sup' and `inf' in \eqref{eqn:defaltspec} by `max' and `min'.

We now have the following counterpart of the monotonicity in Lemma~\ref{lem:ckaltmonotonicity}, which also involves the alternative ECH capacities. To state the result, if $Y$ is a closed three-manifold, and if $f_1,f_2:Y\to\R$ are smooth functions with $f_1<f_2$, define a symplectic four-manifold
\[
M_{f_1,f_2}\eqdef \left\{(s,y)\in\R\times Y\;\big|\;f_1(y)<s<f_2(y)\right\},
\]
with the symplectic form given by the symplectization form $d(e^s\lambda)$.

\begin{lemma}
\label{lem:ckmonotonicity}
(\cite[Lem.\ 4.4]{altspec})
Let $Y$ be a closed three-manifold and let $\lambda$ be a contact form on $Y$. Suppose that $f_1,f_2: Y\to \R$ are smooth functions with $f_1<f_2$ such that the contact forms $e^{f_i}\lambda$ are nondegenerate. Then for any nonnegative integers $k$ and $l$, we have
\begin{equation}
\label{eqn:ckextendedmonotonicity}
c_k\left(Y,e^{f_1}\lambda\right) + c_l^{\op{Alt}}\left(M_{f_1,f_2}\right) \le c_{k+l}\left(Y,e^{f_2}\lambda\right).
\end{equation}
\end{lemma}

This lemma is proved in \cite{altspec} using a neck stretching argument similar to the proof of Lemma~\ref{lem:ckaltmonotonicity}. We remark that for this proof to work, it is important that we used the larger class of almost complex structures $\widetilde{\mathcal{J}}(Y,\lambda)$ in the definition of $c_k$, rather than the smaller class of almost complex structures $\mathcal{J}(Y,\lambda)$.

The special case of Lemma~\ref{lem:ckmonotonicity} where $l=0$ gives
\begin{equation}
\label{eqn:ckmonotonicity}
c_k\left(Y,e^{f_1}\lambda\right) \le c_{k}\left(Y,e^{f_2}\lambda\right).
\end{equation}
This inequality allows us to extend the definition of $c_k$ to degenerate contact forms as follows:

\begin{definition}
\label{def:ckdegen}
Let $Y$ be a closed three-manifold, let $\lambda$ be a contact form on $Y$, and let $k$ be a nonnegative integer. Define
\begin{equation}
\label{eqn:defckdegen}
c_k(Y,\lambda) = \sup\left\{c_k\left(Y,e^f\lambda\right)\right\},
\end{equation}
where the supremum is over smooth functions $f: Y \to (-\infty,0]$ such that the contact form $e^f\lambda$ is nondegenerate.
\end{definition}

Note that because of the Conformality property that we will prove shortly, in \eqref{eqn:defckdegen} we could equivalently take the infimum over smooth functions $f: Y \to [0,\infty)$ such that $e^f \lambda$ is nondegenerate.

%%%%%%%%%%%%%%%%%%%%%%%%%%%%%%%%%%%%%%%%%%%%%%%%%%%%%%%%%%%%%

\subsection{Proofs of basic properties}
\label{sec:ckbasicproofs}

\begin{proof}[Proof of Proposition~\ref{prop:ckproperties}.]
{\em Increasing.\/}
This follows immediately from the definition.

{\em Monotonicity.} This follows from the inequality \eqref{eqn:ckmonotonicity} and Definition~\ref{def:ckdegen}.

{\em Conformality.\/}
This follows the proof of the Conformality property in Proposition~\ref{prop:ckaltproperties}.

{\em $C^0$-Continuity.} This follows from the Conformality and Monotonicity properties.

{\em Spectrality.\/} This follows from the same argument as the proof of the Spectrality property of $c_k^{\op{Alt}}$ in Proposition~\ref{prop:ckaltproperties}.

{\em Disjoint Union.\/} This follows immediately from the definition of $c_k$.

{\em Width Bound.\/} By $C^0$-Continuity, we can assume without loss of generality that the contact forms $\lambda$ and $e^f\lambda$ are nondegenerate. By Lemma~\ref{lem:ckmonotonicity}, we have
\[
c_{k+1}(Y,e^{f}\lambda) \ge c_k(Y,\lambda) + c_1^{\op{Alt}}(M_f). 
\]
By the Monotonicity and Ball properties in Proposition~\ref{prop:ckaltproperties}, we have
\[
\begin{split}
c_1^{\op{Alt}}(M_f) &\ge \sup\left\{c_1^{\op{Alt}}\left(B^4(r)\right) \;\big|\; \mbox{$B^4(r)$ symplectically embeds into $M_f$}\right\}\\
&= \sup\left\{r \;\big|\; \mbox{$B^4(r)$ symplectically embeds into $M_f$}\right\}\\
&= c_{\op{Gr}}(M_f).
\end{split}
\]
The above inequalities imply the Width Bound.
\end{proof}

%%%%%%%%%%%%%%%%%%%%%%%%%%%%%%%%%%%%%%%%%%%%%%

\subsection{Elementary spectral invariants of star-shaped hypersurfaces}
\label{sec:ckmoreproofs}

%%%%%%%%%%%%%%%%%%%%%%%%%%%%%%%%%%%%%%%%%%%%%

We now use results from \S\ref{sec:comp} to prove more properties of the elementary spectral invariants $c_k$, focusing on the special case of star-shaped hypersurfaces in $\R^4$.

We start with the following basic relation between elementary spectral invariants and alternative ECH capacities:

\begin{lemma}[\cite{altspec}]
\label{lem:ckaltck}
Let $(X,\lambda)$ be a four-dimensional Liouville domain with boundary $Y$. Then for each nonnegative integer $k$, we have
\[
c_k(Y,\lambda|_Y) \le c_k^{\op{Alt}}(X,\lambda).
\]
\end{lemma}

\begin{proof}
This is part of \cite[Thm.\ 1.14]{altspec}, and it is proved in \cite[\S4.2]{altspec} by a neck stretching argument similar to the proof of Lemma~\ref{lem:ckaltmonotonicity}.
\end{proof}

We now have the following variant of Proposition~\ref{prop:enhanced}:

\begin{lemma}
\label{lem:enhancedvariant}
Let $Y\subset\R^4$ be a star-shaped hypersurface, and let $k$ be a nonnegative integer. Then:
\begin{description}
\item{(a)}
$c_k(Y) < \infty$.
\item{(b)}
If the contact form on $Y$ is nondegenerate, then there exist orbit sets $\alpha_+$ and $\alpha_-$ in $Y$ such that
\begin{gather}
\label{eqn:ev1}
c_k(Y)  =\mathcal{A}(\alpha_+),\\
\label{eqn:ev2}
I(\alpha_+) - I(\alpha_-) \ge 2k.
\end{gather}
\end{description}
\end{lemma}

\begin{proof}
(a) Let $X$ be the star-shaped domain bounded by $Y$. Since $X$ is compact, if $r$ is sufficiently large then $X\subset B^4(r)$. Then by Lemma~\ref{lem:ckaltck} and the Monotonicity and Ball properties in Proposition~\ref{prop:ckaltproperties}, it follows that
\[
c_k(Y) \le c_k^{\op{Alt}}(X) \le c_k^{\op{Alt}}(B^4(r)) < \infty.
\]

(b) 
This parallels the proof of Proposition~\ref{prop:enhanced}. Here one needs to know that if $J\in\widetilde{\mathcal{J}}(Y,\lambda)$ is generic and if $u\in\mathcal{M}^J(\R\times Y)$ has no multiply covered components, and if $\alpha_+$ and $\alpha_-$ are the orbit sets resulting from the positive and negative punctures of $u$, then $\mathcal{M}^J(\R\times Y)$ is a smooth manifold near $u$ of dimension $\le I(\alpha_+) - I(\alpha_-)$. This fact follows from the ECH index inequality in \cite{pfh2,ir,bn}.
\end{proof}

Recall that the alternative ECH capacities of a convex toric domain were computed in Theorem~\ref{thm:ckaltconvex}. We now show that the elementary spectral invariants of the boundary of a convex toric domain are given by the same formula:

\begin{theorem}
\label{thm:ckconvex}
Let $X_\Omega\subset\R^4$ be a convex toric domain. Then
\[
c_k(\partial X_\Omega) = c_k^{\op{Alt}}(X_\Omega).
\]
\end{theorem}

\begin{proof}
We know from Lemma~\ref{lem:ckaltck} that
\[
c_k(\partial X_\Omega) \le c_k^{\op{Alt}}(X_\Omega).
\]

To prove the reverse inequality, fix a nonnegative integer $k$, let $\epsilon>0$ be small with respect to $k$, and let $X$ be a nondegenerate perturbation of $X_\Omega$ as in the proof of Theorem~\ref{thm:ckaltconvex} in \S\ref{sec:ckaltconvex}. Write $Y=\partial X$. Let $\alpha_+$ and $\alpha_-$ be orbit sets provided by Lemma~\ref{lem:enhancedvariant}(b), and let $\Lambda_+$ and $\Lambda_-$ be the convex integral paths corresponding to them as in \S\ref{sec:ckaltconvex}. By the left inequality in \eqref{eqn:Iconvex} applied to $\alpha_-$, we have
\[
I(\alpha_-) \ge 2\op{Area}(P_{\Lambda_-}) + a(\Lambda_-) + b(\Lambda_-) \ge 0.
\]
So by \eqref{eqn:ev2} we have $I(\alpha_+)\ge 2k$.

It then follows from the inequality \eqref{eqn:ILambda} applied to $\alpha_+$ that
\[
\mathcal{L}(\Lambda_+) \ge k+1.
\]
On the other hand, by equation \eqref{eqn:ALambda} applied to $\alpha_+$, we have
\[
\mathcal{A}(\alpha_+) = \ell_\Omega(\Lambda_+) + O(\epsilon).
\]
Combining this with \eqref{eqn:ev1} and taking the limit as $\epsilon\to 0$, we conclude that $c_k(\partial X_\Omega)$ satisfies the lower bound in \eqref{eqn:ckaltconvexlb}. Thus by the inequality \eqref{eqn:roundingcorners}, the spectral invariant $c_k(\partial X_\Omega)$ is greater than or equal to the formula for $c_k^{\op{Alt}}(X_\Omega)$ in Theorem~\ref{thm:ckaltconvex}.
\end{proof}

\begin{proof}[Proof of Proposition~\ref{prop:ckellipsoid}]
This follows from Theorem~\ref{thm:ckconvex} and Corollary~\ref{cor:ckaltellipsoid}.
\end{proof}

\subsection{The Weyl law for star-shaped hypersurfaces}
\label{sec:ssweyl}

We now prove the Weyl law in Theorem~\ref{thm:ckweyl} for the special case of star-shaped hypersurfaces:

\begin{theorem}
\label{thm:ssweyl}
Let $Y\subset\R^4$ be a star-shaped hypersurface, and let $X\subset \R^4$ be the domain that it bounds. Then
\[
\lim_{k\to \infty}\frac{c_k(Y)^2}{k} = 4\op{vol}(X).
\]
\end{theorem}

Note that by equation \eqref{eqn:symplecticvolume} and Stokes theorem, $\op{vol}(X)$ is half the contact volume of $Y$.

To prove Theorem~\ref{thm:ssweyl}, we first show that the ``lower bound half'' of the Weyl law holds for any contact three-manifold. (In general the upper bound half is the hard part.)

\begin{lemma}
\label{lem:weyllb}
(part of \cite[Thm.\ 1.14]{altspec})
Let $Y$ be a closed three-manifold and let $\lambda$ be a contact form on $Y$. Then
\[
\liminf_{k\to \infty}\frac{c_k(Y,\lambda)^2}{k} \ge 2\op{vol}(Y,\lambda).
\]
\end{lemma}

\begin{proof}
Let $R_0>0$. By Lemma~\ref{lem:ckmonotonicity} with $k=0$, $f_1=-R_0$, and $f_2=0$, and switching the letters $k$ and $l$, we have\footnote{Strictly speaking, this application of Lemma~\ref{lem:ckmonotonicity} assumes nondegeneracy of $\lambda$, but Lemma~\ref{lem:ckmonotonicity} extends to the degenerate case by the $C^0$-Continuity property in Proposition~\ref{prop:ckproperties}.}
\[
c_k(Y,\lambda) \ge c_k^{\op{Alt}}(M_{-R_0,0}).
\]
By Proposition~\ref{prop:ckaltlowerweyl} and Stokes theorem, we have
\[
\begin{split}
\liminf_{k\to\infty}\frac{c_k^{\op{Alt}}(M_{-R_0,0})^2}{k} &\ge 4\op{vol}(M_{-R_0,0}) \\
&= 2\left(1-e^{-2R_0}\right)\op{vol}(Y,\lambda).
\end{split}
\]
Since $R_0$ can be arbitrarily large, the lemma follows.
\end{proof}

\begin{proof}[Proof of Theorem~\ref{thm:ssweyl}]
By Lemma~\ref{lem:weyllb}, we just need to show that
\begin{equation}
\label{eqn:lastinequalitytoshow}
c_k(Y) \le 2\sqrt{k\op{vol}(X)} + o(k^{1/2}).
\end{equation}
Choose $r>0$ sufficiently large that $X\subset\op{int}(B^4(r))$. By Lemma~\ref{lem:ckmonotonicity} and Theorem~\ref{thm:ckconvex}, we have
\[
\begin{split}
c_k(Y) + c_l^{\op{Alt}}\left(B^4(r)\setminus X\right) &\le c_{k+l}(\partial B^4(r))\\
&= c_{k+l}^{\op{Alt}}(B^4(r)).
\end{split}
\]
The inequality \eqref{eqn:lastinequalitytoshow} now follows by the calculation in \eqref{eqn:klcalculation}.
\end{proof}

\end{document}